\documentclass[12pt]{article}
\usepackage{amsmath}\usepackage{amssymb}\usepackage{amsfonts}\usepackage{amsthm}
\usepackage{color,url}
\usepackage{dsfont}
\usepackage{setspace}
\usepackage{graphicx}
\usepackage[english]{babel}\usepackage[latin1]{inputenc}\usepackage{times}\usepackage[T1]{fontenc}
\usepackage{epstopdf}
\usepackage{epsfig}
\usepackage{float}
\usepackage{stmaryrd}

\newtheorem{thm}{\bf{Theorem}}[section]
\newtheorem{lem}[thm]{\bf{Lemma}}

\newtheorem{df}[thm]{\bf{Definition}}
\newtheorem{cor}[thm]{\bf{Corollary}}
\newtheorem{rem}[thm]{\bf{Remark}}

\newtheorem{prop}[thm]{\bf{Proposition}}
\newtheorem{fact}[thm]{\bf{Fact}}
\newtheorem{ex}[thm]{\bf{Example}}

\numberwithin{equation}{section}
\numberwithin{thm}{section}

\newcommand{\dom}{\operatorname{dom}}

\newcommand{\ri}{\operatorname{ri}}

\newcommand{\gra}{\operatorname{gra}}
\newcommand{\Span}{\operatorname{span}}
\newcommand{\epi}{\operatorname{epi}}

\newcommand{\ran}{\operatorname{ran}}

\newcommand{\Id}{\operatorname{Id}}

\newcommand{\elimsup}{\operatornamewithlimits{elimsup}}
\newcommand{\eliminf}{\operatornamewithlimits{eliminf}}

\newcommand{\elim}{\operatornamewithlimits{elim}}

\newcommand{\argmin}{\operatornamewithlimits{argmin}}

\newcommand{\Prox}{\operatorname{Prox}}

\newcommand{\R}{\operatorname{\mathbb{R}}}
\newcommand{\N}{\operatorname{\mathbb{N}}}

\newcommand{\C}{\operatorname{\mathcal{B}}}
\newcommand{\RR}{\ensuremath{\mathbb R}}

\newcommand{\OR}{\operatorname{\overline{\mathbb{R}}}}

\title{Epiconvergence, the Moreau envelope and generalized linear-quadratic functions}
\author{C. Planiden\thanks{Chayne Planiden, Department of Mathematics, University of British Columbia Okanagan, Kelowna, B.C. V1V 1V7, Canada. Email: chayne.planiden@ubc.ca.}\and X. Wang\thanks{Xianfu Wang, Department of Mathematics, University of British Columbia Okanagan, Kelowna, B.C. V1V 1V7, Canada. Email: shawn.wang@ubc.ca.}}
\date{November 24, 2017}
\begin{document}

\maketitle\author
\begin{center}
{\em Dedicated to Dr. Roger J-B Wets, one of the pioneers of epiconvergence.}
\vspace{12pt}
\end{center}
\setcounter{page}{1}\pagenumbering{arabic}

\begin{abstract}
This work introduces the class of generalized linear-quadratic functions, constructed using maximally monotone symmetric linear relations. Calculus rules and properties of the Moreau envelope for this class of functions are developed. In finite dimensions, on a metric space defined by Moreau envelopes, we consider the epigraphical limit of a sequence of quadratic functions and categorize the results. We explore the question of when a quadratic function is a Moreau envelope of a generalized linear-quadratic function; characterizations involving nonexpansiveness and Lipschitz continuity are established. This work generalizes some results by Hiriart-Urruty and by Rockafellar and Wets.
\end{abstract}



\section{Introduction}\label{sec:intro}

The Moreau envelope is a well-established and extensively researched function that emerged in the 1960s \cite{proximite}. It is of great use in optimization due to its regularizing properties, differentiability and coincidence of the minimizers of the objective function in the convex setting. This work continues the investigation into Moreau envelopes in finite dimensions, from the perspective of the generalized linear-quadratic objective function. The long-range reason for studying Moreau envelopes in general is, as alluded to in \cite{genhess}, that if we had a sufficient level of understanding about their properties, it would likely facilitate the development of minimization methods for Moreau envelopes, and therefore for their associated (nonsmooth) objective functions as well. In this work we focus on generalized linear-quadratic functions, because it is a class of functions that has enough structure to secure solid results that do not require overly restrictive conditions, but allows us to obtain results that are useful for a wide range of functions. We define a metric space whose distance function is constructed using Moreau envelopes, with the intention of exploring the epiconvergence of a sequence of quadratic functions. (See \cite{attouch1991quantitative} for more on epidistance between functions.) The idea of studying epiconvergence of convex functions via the Moreau envelope is due to Attouch \cite{attouch1984} and Attouch-Wets \cite{attouch1991quantitative}. Several classes of functions can arise at the limit; these results are classified and illustrated. Then we approach the relationship between Moreau envelopes and quadratics from the opposite direction, asking under what conditions a given quadratic function is a Moreau envelope of another function, and whether said other function can be determined explicitly.\par The linear relation is also a useful tool in functional analysis, notably documented and developed in \cite{cross1998multivalued}, with more recent expansion in \cite{resaverage,monolinrel,bauschke2010borwein,liangjin}. This paper continues to develop the theory of monotone linear relations, in particular for the class of generalized linear-quadratic functions. Such functions arise, for example, in the determination of the existence of a Hessian for the Moreau envelope \cite{lemarechal1997practical,genhess}. In \cite[Theorem 3.9]{genhess} Rockafellar and Poliquin showed that a function does not have to be finite in order for its Moreau envelope to have a Hessian; it suffices that the second-order epiderivative of the function be a generalized linear-quadratic function. The existence of a Hessian is of interest since it is needed in order to do a second-order expansion of the Moreau envelope function, which leads to a second-order approximation of its objective function. Several properties and characterizations for the class of generalized linear-quadratic functions are provided in this work and we demonstrate that it is useful and convenient to work in the setting of generalized linear-quadratic functions when considering matters of epiconvergence. In this paper, we show
\begin{enumerate}
\item[(i)]that monotone linear relations provide a unified framework for generalized linear-quadratic functions;
\item[(ii)]that the Fenchel conjugate of every generalized linear-quadratic convex function can be written in terms of the set-valued inverse of a monotone linear relation;
\item[(iii)]that a function is convex generalized linear-quadratic if and only if its Moreau envelope is convex quadratic;
\item[(iv)]the relationship between the set-valued inverse of a linear mapping and its Moore-Penrose inverse.
\end{enumerate}
We also establish calculus rules for the set of generalized linear-quadratic functions, and we generalize the result of Rockafellar \cite[p. 136]{convanalrock} and that of Hiriart-Urruty deconvolution \cite[Example 2.7]{hiriart1994deconvolution} from symmetric positive definite matrices to maximally monotone symmetric linear relations.\par The rest of this paper is organized as follows. Section \ref{sec:prelim} contains notation, definitions, and basic properties of Moreau envelopes, epiconvergence and monotone operators. In Section \ref{sec:epi}, we discuss epigraphical limits of linear-quadratic functions in one dimension. Several illustrative examples are presented, with graphs showing the limiting behaviour of the Moreau envelope for sequences of quadratic functions. Section \ref{sec:epi2} contains the principal matter of this work; properties, characteristics and results on epiconvergence of generalized linear-quadratic functions on finite-dimensional space. Topics include symmetry, maximal monotonicity, nonexpansiveness, the subdifferential, sum, difference and infimal convolution rules, the adjoint, the set-valued and Moore--Penrose inverses and the Fenchel conjugate. This section includes characterizations of Moreau envelopes of generalized linear-quadratic functions. In Section \ref{sec:semi}, we give applications of these results and we develop a calculus of the class of generalized linear-quadratic functions. Applications are to the seminorm function, the least squares problem and the limit of a sequence of generalized linear-quadratic functions. Section \ref{sec:conc} makes concluding remarks.

\section{Preliminaries}\label{sec:prelim}

This section collects several definitions and facts from previous works, that we will use later in proving our main results. For proof of the facts, we refer the reader to the corresponding citations.

\subsection{Notation}

All functions in this work are defined on $\R^n,$ Euclidean space equipped with inner product $\langle x,y\rangle=\sum_{i=1}^nx_iy_i$ and induced norm $\|x\|=\sqrt{\langle x,x\rangle}.$ The extended real line $\R\cup\{\infty\}$ is denoted $\overline{\R}.$ We use $\Gamma_0(\R^n)$ to represent the set of proper, convex, lower semicontinuous (lsc) functions on $\R^n.$ The identity operator is denoted $\Id.$ Pointwise convergence is denoted $\overset{p}\rightarrow,$ graphical convergence $\overset{g}\rightarrow$ and epiconvergence $\overset{e}\rightarrow.$ The function $\frac{1}{2}\|\cdot\|^2$ is denoted $q.$ We use $N_C(x)$ to represent the normal cone to $C$ at $x,$ as defined in \cite{rockwets}. The relative interior of a set $A$ is denoted $\ri A.$ The domain and the range of an  operator $A$ are denoted $\dom A$ and $\ran A,$ respectively. On $\OR,$ where necessary we use inf-addition and accept the convention $\infty-\infty=\infty$ (see \cite[p. 15]{rockwets}). We use $S_n,$ $S_n^+$ and $S_n^{++}$ to denote the sets of symmetric, positive semidefinite and positive definite matrices, respectively. The \emph{graph} of an operator $A:\R^n\rightrightarrows\R^n$ is defined
$$\gra A=\{(x,x^*):x^*\in Ax\}.$$Its set-valued inverse $A^{-1}:\R^n\rightrightarrows\R^n$ is defined by the graph
$$\gra A^{-1}=\{(x^*,x):x^*\in Ax\}.$$
For any function $f:\R^n\rightarrow\OR,$ the function $f^*:\R^n\rightarrow\OR$ defined by
$$f^*(x^*)=\sup\limits_{x\in\R^n}[\langle x^*,x\rangle-f(x)]$$
is called the \emph{Fenchel conjugate} of $f.$ The \emph{Fenchel subdifferential} of $f$ at $x\in\dom f$ is the set
$$\partial f(x)=\{x^*:f(y)\geq f(x)+\langle x^*,y-x\rangle~\forall y\in\R^n\}.$$ For any $x\not\in\dom f,$ $\partial f(x)=\varnothing.$

\subsection{Moreau envelopes, proximal mappings and their properties}

We work with Moreau envelopes of functions throughout this paper.

\begin{df}\label{df:Moreau}
For a proper, lsc function $f:\R^n\rightarrow\OR,$ the \emph{Moreau envelope} of $f$ is denoted $e_rf$
and defined
$$e_rf(x)=\inf\limits_{y\in\R^n}\left\{f(y)+\frac{r}{2}\|y-x\|^2\right\}.$$
The vector $x$ is called the \emph{prox-centre} and the scalar $r\geq0$ is called the \emph{prox-parameter}. The associated \emph{proximal mapping} is the set of all points at which the above infimum is attained, denoted $\Prox_f^r:$
$$\Prox_f^r(x)=\argmin\limits_{y\in\R^n}\left\{f(y)+\frac{r}{2}\|y-x\|^2\right\}.$$
\end{df}
\begin{lem}\label{fenchmorlem}
For any proper function $f:\R^n\rightarrow\OR,$
\begin{equation}\label{e1}
e_rf(x)=\frac{r}{2}\|x\|^2-g^*(rx),
\end{equation}
where $g(x)=f(x)+\frac{r}{2}\|x\|^2.$
\end{lem}
\begin{proof}
We have
\begin{align*}
e_rf(x)&=\inf\limits_y\left\{f(y)+\frac{r}{2}\|y-x\|^2\right\}\\
&=-\sup\limits_y\left\{-f(y)-\frac{r}{2}(\|y\|^2-2\langle x,y\rangle+\|x\|^2)\right\}\\
&=\frac{r}{2}\|x\|^2-\sup\limits_y\left\{\langle rx,y\rangle-\left(f(y)+\frac{r}{2}\|y\|^2\right)\right\}\\
&=\frac{r}{2}\|x\|^2-g^*(rx).
\end{align*}
\end{proof}
\begin{fact}\emph{\cite[Example 23.3]{convmono}}\label{resolventfact} In the case of a convex function $f,$ an alternate representation of the proximal mapping makes use of the \emph{resolvent} of the subdifferential of $f,$ which also provides a conversion to the proximal mapping with prox-parameter 1:
$$\Prox_f^r=\left(\Id+\frac{1}{r}\partial f\right)^{-1}=\Prox_{\frac{1}{r}f}^1.$$
\end{fact}
An alternate expression for the Moreau envelope is reached through infimal convolution:\begin{equation}\label{infconvmor}e_rf=f\oblong\frac{r}{2}\|\cdot\|^2=f\oblong(rq).\end{equation}
\begin{fact}\emph{\cite[Theorem 16.2]{convmono}}\label{fermat} For any proper, lsc function $f:\R^n\rightarrow\OR$ and any $r>0,$ we have
\begin{align*}
p\in\Prox_f^r(x)&\Rightarrow0\in\partial f(p)+r(p-x)\\
&\Leftrightarrow0\in\frac{1}{r}\partial f(p)+p-x.
\end{align*}If in addition $f$ is convex, then the first implication above becomes a two-way implication:$$p\in\Prox_f^r(x)\Leftrightarrow0\in\partial f(p)+r(p-x).$$
\end{fact}
\begin{prop}[Calculus of Moreau envelopes]\label{calcenv}
For any function $f:\R^n\rightarrow\OR,$ $r>0,$ $v\in\R^n,$ $c\in\R,$ the following hold:
\begin{itemize}
\item[(i)]$e_r(f+c)=e_rf+c;$
\item[(ii)]$e_rf=re_1(f/r);$
\item[(iii)]$e_r(f(\cdot-c))=(e_rf)(\cdot-c);$
\item[(iv)]$e_1f=q-(f+q)^*;$
\item[(v)]$e_1(f+\langle\cdot,v\rangle)=e_1f(\cdot-v)+\langle\cdot,v\rangle-q(v);$
\item[(vi)]$(e_rf)^*=f^*+q/r.$
\end{itemize}
\end{prop}
\begin{proof}
(i) This is seen directly as a property of the infimum: for any function $g$ and any $c\in\R,$ $\inf\{g(x)+c\}=\inf\{g(x)\}+c.$\\

\noindent(ii) See \cite[Proposition 12.22]{convmono}.\\

\noindent(iii) Let $z=y-c.$ Then\begin{align*}
e_r(f(\cdot-c))(x)&=\inf\limits_{y\in\R^n}\left\{f(y-c)+\frac{r}{2}\|y-x\|^2\right\}\\
&=\inf\limits_{z\in\R^n}\left\{f(z)+\frac{r}{2}\|z-(x-c)\|^2\right\}\\
&=(e_rf)(x-c).
\end{align*}
(iv) This is Lemma \ref{fenchmorlem} with $r=1.$\\

\noindent(v) Consider the left-hand side of statement (v) first. Applying statement (iv) to $f+\langle\cdot,v\rangle,$ we have
$$e_1(f+\langle\cdot,v\rangle)=q-(f+\langle\cdot,v\rangle+q)^*.$$ Applying \cite[Proposition 13.20(iii)]{convmono} to the function $f+q$ with $y=0$ and $\alpha=0,$ we have
\begin{equation}\label{eq:sameasotherone}
e_1(f+\langle\cdot,v\rangle)=q-[f(\cdot-v)+q(\cdot-v)]^*.
\end{equation}Now consider the right-hand side of statement (v). Applying statement (iv) to $f(\cdot-v),$ we have
\begin{align*}
e_1(f(\cdot-v))&=q(\cdot-v)-[f(\cdot-v)+q(\cdot-v)]^*,\\
&=q-[f(\cdot-v)+q(\cdot-v)]^*-\langle\cdot,v\rangle+q(v),\\
e_1(f(\cdot-v))+\langle\cdot,v\rangle-q(v)&=q-[f(\cdot-v)+q(\cdot-v)]^*,
\end{align*}which is the same as \eqref{eq:sameasotherone}.\\

\noindent(vi) By \cite[Proposition 13.21(iii)]{convmono} with $g=rq,$ we have $(f\oblong rq)^*=f^*+(rq)^*.$ By \eqref{infconvmor} we have $(e_rf)^*=(f\oblong rq)^*$ and by \cite[Example 13.4]{convmono} we have $(rq)^*=q/r.$
\end{proof}
\begin{prop}\label{nablaprop}
Let $f\in\Gamma_0(\R^n).$ Then $f$ is prox-bounded with threshold 0, $\Prox_f^r$ is single-valued and continuous, and $e_rf$ is convex and continuously differentiable. Moreover, the following properties hold.
\begin{align*}
\mbox{(i)}&&~ e_rf(x)+e_{\frac{1}{r}}f^*(rx)&=\frac{r}{2}\|x\|^2;\\
\mbox{(ii)}&&~ \nabla e_rf(x)&=r[x-\Prox_f^r(x)];\\
\mbox{(iii)}&&~ \nabla e_rf^*(x)&=\Prox_f^{\frac{1}{r}}(rx);&\\
\mbox{(iv)}&&~ \Prox_f^r(x)&=\nabla g(x)\mbox{ where }g(x)=\frac{1}{r}\left[e_{\frac{1}{r}}f^*(rx)\right];\\
\mbox{(v)}&&~ \Prox_{f^*}^r(x)&=x-\frac{1}{r}\Prox_f^{\frac{1}{r}}(rx).&
\end{align*}
\end{prop}
\begin{proof}
The proof that $f$ has threshold $0,$ $\Prox_f^r$ is single-valued and continuous, and $e_rf$ is convex and continuously differentiable is found in \cite[Theorem 2.26]{rockwets}.\\

\noindent(i) See \cite[Example 11.26]{rockwets}.\\

\noindent(ii) See \cite[Theorem 2.26]{rockwets}.\\

\noindent(iii) Replacing $f$ with $f^*$ in part (i) and using the fact that $f^{**}=f,$ we have
$$e_rf^*(x)+e_{\frac{1}{r}}f(rx)=\frac{r}{2}\|x\|^2.$$
Differentiating both sides and rearranging yields
$$\nabla e_rf^*(x)=rx-\nabla e_{\frac{1}{r}}f(rx).$$
We substitute $z=rx,$ then use part (ii) and the chain rule to get
\begin{align*}
\nabla e_rf^*(x)&=rx-\nabla_xe_{\frac{1}{r}}f(z)\\
&=rx-\nabla_ze_{\frac{1}{r}}f(z)\nabla_xz\\
&=rx-\frac{1}{r}[z-\Prox_f^{\frac{1}{r}}(z)]r\\
&=\Prox_f^{\frac{1}{r}}(rx).
\end{align*}
(iv) See \cite[Exercise 11.27]{rockwets}.\\

\noindent(v) Replacing $f$ with $f^*$ in part (iv), we have
$$\Prox_{f^*}^r(x)=\nabla g(x)\mbox{ where }g(x)=\frac{1}{r}\left[e_{\frac{1}{r}}f(rx)\right].$$
Substituting $z=rx,$ then applying part (ii) and the chain rule yields
\begin{align*}
\nabla g(x)&=\frac{1}{r}\nabla_z\left(e_{\frac{1}{r}}f(z)\right)\nabla_xz\\
&=\frac{1}{r}\left[\frac{1}{r}\left(z-\Prox_f^{\frac{1}{r}}(z)\right)\right]r\\
&=\frac{1}{r}\left(rx-\Prox_f^{\frac{1}{r}}(rx)\right)\\
&=x-\frac{1}{r}\Prox_f^{\frac{1}{r}}(rx).
\end{align*}
\end{proof}

\subsection{Epiconvergence and the Attouch-Wets metric}

Epiconvergence plays a fundamental role in optimization and variational analysis, see \cite{attouch1984, vanderwerff, beer93, burkeho, rockwets}.
\begin{df}\label{df:epi}
For any sequence $\{f_k\}$ of functions on $\R^n,$ the \emph{lower epilimit $\eliminf_kf_k$} is the function having as its epigraph the outer limit of the sequence of sets $\epi f_k:$
$$\epi(\eliminf\limits_kf_k)=\limsup\limits_k(\epi f_k).$$ The \emph{upper epilimit $\elimsup_kf_k$} is the function having as its epigraph the inner limit of the sets $\epi f_k:$
$$\epi(\elimsup\limits_kf_k)=\liminf\limits_k(\epi f_k).$$ When these two functions coincide, the \emph{epilimit $\elim_kf_k$} is said to exist: $$\elim_kf_k=\eliminf_kf_k=\elimsup_kf_k.$$ In this event, the functions are said to \emph{epiconverge} to $f,$ symbolized by $f_k\overset{e}\rightarrow f.$ Thus,
$$f_k\overset{e}\rightarrow f\Leftrightarrow\epi f_k\overset{g}\rightarrow\epi f.$$
\end{df}
\begin{df}
Let $f:\R^n\rightarrow\OR$ be proper and lsc. If there exists $r\geq0$ such that $e_rf(x)>-\infty$ for some $x,$ then $f$ is said to be \emph{prox-bounded}. The infimum of all such $r$ is called the \emph{threshold of prox-boundedness} of $f.$
\end{df}
\begin{df}
A sequence of functions $\{f_k\}$ on $\R^n$ is \emph{eventually prox-bounded} if there exists $r\geq0$ such that $\liminf_{k\rightarrow\infty}e_rf_k(x)>-\infty$ for some $x.$ The infimum of all such $r$ is the \emph{threshold of eventual prox-boundedness} of the sequence.
\end{df}
 There is an important relationship among epiconvergence of proper lsc functions, pointwise convergence and uniform convergence of their Moreau envelopes, as the following fact outlines.
\begin{fact}\label{fact:convergences}\emph{\cite[Theorem 7.37]{rockwets}}
For proper, lsc functions $f_k$ and $f,$ the following are equivalent:
\begin{itemize}
\item[(i)] the sequence $\{f_k\}$ is eventually prox-bounded and $f_k\overset{e}\rightarrow f;$
\item[(ii)] $f$ is prox-bounded and $e_rf_k\overset{p}\rightarrow e_rf$ for all $r\in(\varepsilon,\infty),$ $\varepsilon>0.$
\end{itemize} Then the pointwise convergence of $e_rf_k$ to $e_rf$ for $r>0$ sufficiently large is uniform on all bounded subsets of $\R^n,$ hence yields continuous convergence and epiconvergence as well, and indeed $e_{r_k}f_k$ converges in all these ways to $e_rf$ whenever $r_k\rightarrow r\in(\bar{r},\infty),$ where $\bar{r}$ is the threshold of eventual prox-boundedness. If $f_k$ and $f$ are convex, then $\bar{r}=0$ and condition (ii) can be replaced by
\begin{itemize}
\item[(ii)] $e_rf_k\overset{p}\rightarrow e_rf$ for some $r>0.$
\end{itemize}
\end{fact}
Epitoplogy is metrizable; we use the following distance function.
\begin{df}[Attouch-Wets metric]\label{attouchwetsmetric}
Let $r>0.$ For $f,g\in\Gamma_0(\R^n),$ define the distance function
$$d(f,g)=\sum\limits_{i=1}^\infty\frac{1}{2^i}\frac{\sup_{\|x\|\leq i}|e_rf(x)-e_rg(x)|}{1+\sup_{\|x\|\leq i}|e_rf(x)-e_rg(x)|}.$$
\end{df}
\begin{fact}\emph{\cite[Proposition 3.5]{planwang2016}}
The space $(\Gamma_0(\R^n),d)$ is a complete metric space.
\end{fact}
\begin{fact}\emph{\cite[Theorem 25.7]{convanalrock}}\label{fact:nablaconvex}
Let $C$ be an open convex set and $f$ be a convex function that is finite and differentiable on $C.$ Let $\{f_k\}_{k\in\N}$ be a sequence of convex functions finite and differentiable on $C$ such that $\lim_{k\to\infty}f_k(x)=f(x)$ for every $x\in C.$ Then$$\lim\limits_{k\to\infty}\nabla f_k(x)=\nabla f(x)~\forall x\in C.$$In fact, the mappings $\nabla f_k$ converge to $\nabla f$ uniformly on every closed bounded subset of $C.$
\end{fact}

\subsection{Monotone operators and resolvents}

In this section, we list a number of facts involving monotonicity, maximal monotonicity and cyclic monotonicity.
\begin{df}
An operator $A:\R^n\rightrightarrows\R^n$ is \emph{monotone} if $\langle x^*-y^*,x-y\rangle\geq0~\forall (x,x^*),(y,y^*)\in\gra A.$ The monotone operator $A$ is \emph{maximally monotone} if there does not exist a monotone operator that contains $A.$
\end{df}
\begin{df}
The \emph{resolvent} $J_A$ of a mapping $A:\R^n\rightrightarrows\R^n$ is defined
$$J_A=(\Id+A)^{-1}.$$
\end{df}
\begin{fact}\emph{\cite[Lemma 12.14]{rockwets}\label{rockwets12.14}}
Every mapping $A:\R^n\rightrightarrows\R^n$ obeys the identity
$$\Id-(\Id+A)^{-1}=(\Id+A^{-1})^{-1}.$$
\end{fact}
\begin{fact}\emph{\cite[Lemma 12.12]{rockwets}\label{rockwets12.12}}
Let $A:\R^n\rightrightarrows\R^n$ be monotone, $\lambda>0.$ Then $(\Id+\lambda A)^{-1}$ is monotone and nonexpansive. Moreover, $A$ is maximally monotone if and only if $\dom(\Id+\lambda A)^{-1}=\R^n.$ In that case, $(\Id+\lambda A)^{-1}$ is maximally monotone as well, and it is a single-valued mapping from all of $\R^n$ into itself.
\end{fact}
\begin{fact}\emph{\cite[Proposition 23.7]{convmono}\label{nonexpequivmono}}
Let $D$ be a nonempty subset of $\R^n,$ $T:D\rightarrow\R^n,$ $A=T^{-1}-\Id.$ Then $T$ is firmly nonexpansive if and only if $A$ is monotone.
\end{fact}
\begin{fact}\emph{\cite[Theorem 6.6]{fitz}\label{cyclic}}
Let $T:\R^n\rightarrow\R^n.$ Then $T$ is the resolvent of the maximally cyclically monotone operator $A:\R^n\rightrightarrows\R^n$ if and only if $T$ has full domain, $T$ is firmly nonexpansive, and for every set of points $\{x_1,\ldots,x_m\}$ where the integer $m\geq2$ and $x_{m+1}=x_1,$ one has
$$\sum\limits_{i=1}^m\langle x_i-Tx_i,Tx_i-Tx_{i+1}\rangle\geq0.$$
\end{fact}
\begin{fact}\emph{\cite[Theorem 22.14]{convmono}\label{maxcyc}}
Let $A:\R^n\rightrightarrows\R^n.$ Then $A$ is maximally cyclically monotone if and only if there exists $f\in\Gamma_0(\R^n)$ such that $A=\partial f.$
\end{fact}
\begin{fact}\emph{(Baillon-Haddad Theorem) \cite[Corollary 10]{baillonhaddad}\label{baillon}}
Let $\varphi$ be a convex $\mathcal{C}^1$ function on $\R^n.$ Let $A=\nabla\varphi.$ If $A$ is $L$-Lipschitz, then
$$\langle Au-Av,u-v\rangle\geq\frac{1}{L}\|Au-Av\|^2\qquad\forall u,v\in\R^n.$$ Hence, $\frac{A}{L}=\nabla\left(\frac{\phi}{L}\right)$ is firmly nonexpansive and 1-Lipschitz.
Consequently, $\frac{A}{L}$ is a proximal mapping: $$\frac{A}{L}=\Prox_g^1\mbox{ for some } g\in\Gamma_0(\R^n).$$
\end{fact}

\section{Epigraphical limits of quadratic functions on $\R$}\label{sec:epi}

One of the main objectives of this paper is to present epiconvergence properties of generalized linear-quadratic functions and their Moreau envelopes. For the first set of results, we focus on quadratic functions on $\R.$ This serves to show the variety of situations that can arise at the epigraphical limit of a sequence of quadratic functions. Then in the next section, we concentrate on the expansion to $\R^n.$
\begin{thm}\label{exquad}
For all $k\in\N,$ let $a_k,b_k,c_k\in\R$ with $a_k\geq0,$ so that $$F=\{f_k(x)=a_kx^2+b_kx+c_k\}_{k=1}^\infty\subseteq\Gamma_0(\R).$$ Then for $r>0$ we have
\begin{equation}\label{eq:er}e_rf_k(x)=\frac{a_kr}{2a_k+r}x^2+\frac{b_kr}{2a_k+r}x+c_k-\frac{b_k^2}{2(2a_k+r)}.\end{equation}
Moreover, letting $k\rightarrow\infty$ and $f_k\overset{e}\rightarrow f,$ we have the following trichotomy.
\begin{itemize}
\item[(i)] If $f\equiv\infty,$ then $e_rf\equiv\infty.$
\item[(ii)] If $f(x)=-\infty$ for some $x,$ then $e_rf\equiv-\infty.$
\item[(iii)] If $f$ is proper, then $e_rf$ is of the form $arx^2+bx+c$ with $a\geq0.$ This is true even in the case where $a_k\to\infty$ and $f(x)=\iota_{\{b\}}(x)+c.$
\end{itemize}
\end{thm}
\begin{proof}
The Moreau envelope is not defined for improper functions such as those of parts (i) and (ii), but if we consider the same definition valid for improper functions, it is clear that in part (i) we have $e_rf\equiv\infty$ and in part (ii) we have $e_rf\equiv-\infty.$ For part (iii), we want to consider the Moreau envelope at the limit of the sequence
\begin{align*}
e_rf_k(x)&=\inf\limits_{y\in\R}\left\{f_k(y)+\frac{r}{2}(y-x)^2\right\}\\
&=\inf\limits_{y\in\R}\left\{\left(a_k+\frac{r}{2}\right)y^2+(b_k-rx)y+c_k+\frac{r}{2}x^2\right\}.
\end{align*}
The infimand above is a strictly convex quadratic function, so its minimum can be found by setting the derivative equal to zero and finding critical points. This yields the minimizer $y=\frac{rx-b_k}{2a_k+r},$ which gives
\begin{align*}
e_rf_k(x)&=\left(a_k+\frac{r}{2}\right)\frac{(rx-b_k)^2}{(2a_k+r)^2}+(b_k-rx)\frac{rx-b_k}{2a_k+r}+c_k+\frac{r}{2}x^2\\
&=\frac{a_kr}{2a_k+r}x^2+\frac{b_kr}{2a_k+r}x+c_k-\frac{b_k^2}{2(2a_k+r)}.
\end{align*}
As expected, $e_rf_k(x)\in\Gamma_0(\R)$ for all $k,$ since the quadratic coefficient is nonnegative. Now consider the sequence $f_k\overset{e}\rightarrow f.$ By Fact \ref{fact:convergences}, we need only consider the pointwise convergence of the sequence $\{e_rf_k\}_{k\in\N}.$
Since $e_rf(x)$ is finite for all $x,$ evaluating \eqref{exquad} at $x=0$ and taking the limit as $k\to\infty$ gives us that the constant coefficient $c_k-b_k^2/[2(2a_k+r)]$ converges to some $c\in\R.$ We know that $e_rf_k$ is differentiable for all $k$ by Proposition \ref{nablaprop}, so $\nabla e_rf_k\to\nabla e_rf$ by Fact \ref{fact:nablaconvex}.  Thus, differentiating \eqref{exquad} and evaluating at $x=0,$ we take the limit to find that the linear coefficient $b_kr/(2a_k+r)$ also converges, to some $b\in\R.$ Finally, evaluating the same derivative at $x=1$ and taking the limit, we have that the coefficient $a_kr/(2a_k+r)$ (which is nonnegative for all $k$) converges to $ar$ for some $a\geq0.$
\end{proof}
Theorem \ref{exquad} leads one to ask which convex functions have quadratic functions as their Moreau envelopes. This question is answered by Proposition \ref{quadprop} below.
\begin{prop}\label{quadprop}
On $\R,$ a convex quadratic function $f:\Gamma_0(\R)\rightarrow\overline{\R},$ $f(x)=\alpha x^2+\beta x+\gamma,$ $\alpha\geq0$ is a Moreau envelope of some convex function $g$ where $g$ is either a quadratic function $g(x)=ax^2+bx+c,$ $a\geq0,$ or an indicator function $g(x)=\iota_{\{b\}}(x)+c.$ Specifically, there exists prox-parameter $r>0$ such that the following hold.
\begin{itemize}
\item[(i)] If $0\leq\alpha<r/2,$ then $g(x)=ax^2+bx+c,$ where
$$a=\frac{\alpha r}{r-2\alpha},~~b=\frac{\beta r}{r-2\alpha},~~c=\gamma+\frac{\beta^2}{2(r-2\alpha)}.$$
\item[(ii)] If $\alpha=r/2,$ then $g(x)=\iota_{\{b\}}(x)+c,$ where
$$b=-\frac{\beta}{r},~~c=\gamma-\frac{\beta^2}{2r}.$$
\item[(iii)]If $\alpha>r/2,$ then $\nexists g\in\Gamma_0(\R)$ such that $f=e_rg.$
\end{itemize}
\end{prop}
\begin{proof}
We need to show the form of $g$ such that $f(x)=e_rg(x)~\forall x\in\R$ for any choice of $\alpha\geq0,$ $\beta,\gamma\in\R.$ By Theorem \ref{exquad}, we have that
$$e_rg(x)=\frac{ar}{2a+r}x^2+\frac{br}{2a+r}x+c-\frac{b^2}{2(2a+r)}.$$
We equate the coefficients of $f$ accordingly:
\begin{equation}\label{three}
\alpha=\frac{ar}{2a+r},~~\beta=\frac{br}{2a+r},~~\gamma=c-\frac{b^2}{2(2a+r)}.
\end{equation}
Solving the first of these expressions for $a,$ we find $a=\alpha r/(r-2\alpha).$ Notice that $\alpha=r/2$ is a point of interest.\\

\noindent(i) If $\alpha\in[0,r/2),$ there is a one-to-one correspondence with $a\in[0,\infty).$ Then $b$ and $c$ are found by solving the equations in \eqref{three}.\\

\noindent(ii) If $\alpha=r/2,$ this corresponds to $g(x)=\iota_{\{b\}}(x)+c:$
\begin{align*}
g(x)&=\begin{cases}c,&x=b,\\\infty,&x\neq b,\end{cases}\\
e_rg(x)&=\inf\limits_y\left\{g(y)+\frac{r}{2}(y-x)^2\right\},\\
&=g(b)+\frac{r}{2}(b-x)^2,\\
&=\frac{r}{2}x^2-brx+\frac{r}{2}b^2+c.
\end{align*}
Equating $\beta=-br$ and $\gamma=rb^2/2+c,$ we find that $b=-\beta/r$ and $c=\gamma-\beta^2/(2r).$ Then $f(x)=e_rg(x)$ where $g(x)=\iota_{\{b\}}(x)+c.$\\

\noindent(iii) Let $\alpha>r/2.$ Suppose that $\exists g\in\Gamma_0(\R)$ such that $f=e_rg.$ By Proposition \ref{nablaprop} and Fact \ref{resolventfact}, we have
$$\nabla e_rg(x)=r(\Id-J_{\partial g/r}).$$Since $(\Id-J_{\partial g/r})=J_{(\partial g/r)^{-1}}$ is nonexpansive, $\nabla e_rg$ is $r$-Lipschitz (see also Proposition \ref{1lip}). On the other hand, we have
$$\nabla e_rg(x)=\nabla f(x)=2\alpha x+\beta,$$which is $L$-Lipschitz only if $L\geq2\alpha.$ Hence, $r\geq2\alpha,$ which contradicts the condition that $\alpha>r/2.$ Therefore, there does not exist $g\in\Gamma_0(\R)$ such that $f=e_rg.$
\end{proof}
There are three possible epigraphical limits for the sequence defined in Theorem \ref{exquad} (see Figure \ref{fig:epi}). The first is $\epi(bx+c),$ the case where $a_k\rightarrow0.$ The second is $\epi(ax^2+bx+c),$ the case where $a_k\rightarrow a>0.$ The third is $\epi(\iota_{\{b\}}(x)+c),$ the case where $a_k\to\infty.$
\begin{figure}[H]
\begin{center}\includegraphics[scale=0.2]{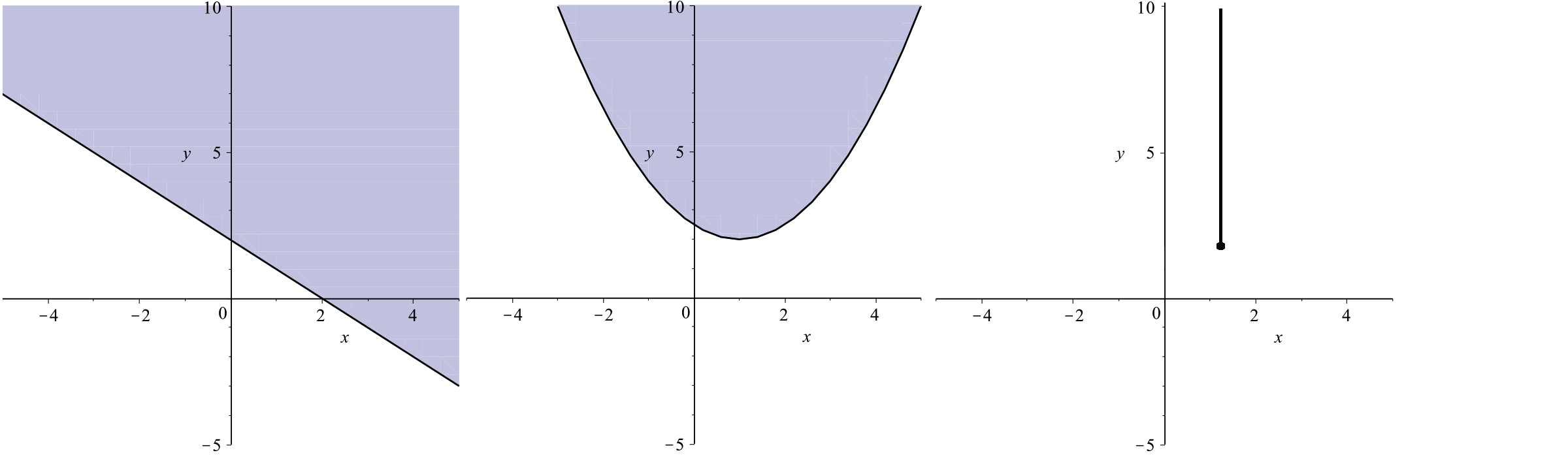}\end{center}
\caption{The three possible general forms of the epigraph of $f(x).$}\label{fig:epi}
\end{figure}
We present three examples here, to illustrate the three possibilities. In all three examples, we set $r=1.$
\begin{ex}
Define $f_k(x)=\left(1+\frac{1}{k}\right)x^2+\left(2+\frac{1}{k}\right)x+\left(1+\frac{1}{k}\right).$ Then $$e_1f_k(x)=\frac{k+1}{3k+2}x^2+\frac{2k+1}{3k+2}x+\frac{2k^2+6k+3}{k(6k+4)}.$$ Letting $k\rightarrow\infty,$ we have $f_k\overset{e}\to f$ with
\begin{align*}
f(x)&=(x+1)^2,\mbox{ and}\\
e_1f(x)&=\frac{1}{3}(x+1)^2.
\end{align*}\end{ex}
Figure \ref{mor1} shows the behaviour of the graphs as a function of $k.$
\begin{figure}[H]
\begin{center}\includegraphics[scale=0.4]{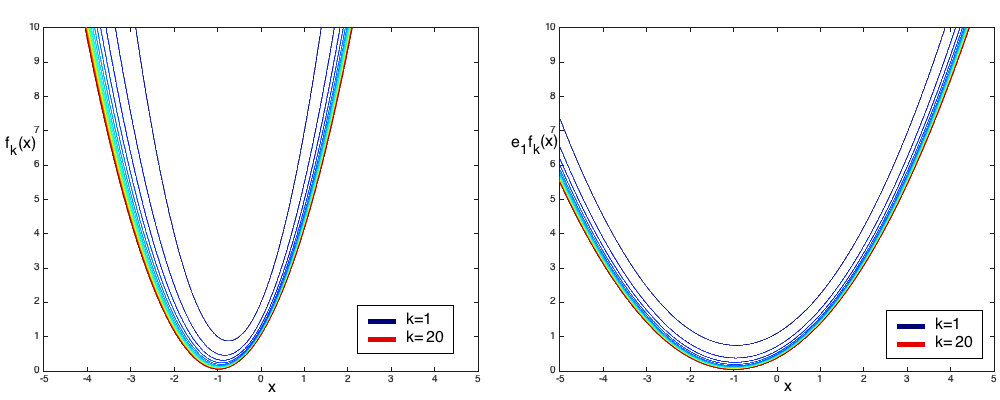}\end{center}
\caption{Left: $f_k(x).$ Right: $e_1f_k(x).$}\label{mor1}
\end{figure}
\begin{ex}Define $g_k(x)=\frac{1}{k}x^2+\left(1+\frac{1}{k}\right)x+\frac{1}{k}.$ Then $$e_1g_k(x)=\frac{1}{k+2}x^2+\frac{k+1}{k+2}x+\frac{-k^2+3}{2k(k+2)}.$$ Letting $k\rightarrow\infty,$ we have $g_k\overset{e}\to g$ with
\begin{align*}
g(x)&=x,\mbox{ and}\\
e_1g(x)&=x-\frac{1}{2}.
\end{align*}\end{ex}
Figure \ref{mor2} shows the behaviour of the graphs as a function of $k.$
\begin{figure}[H]
\begin{center}\includegraphics[scale=0.4]{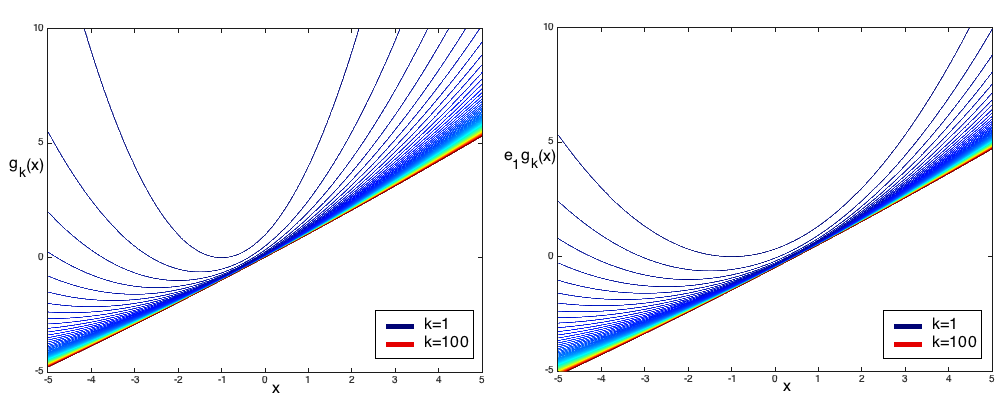}\end{center}
\caption{Left: $g_k(x).$ Right: $e_1g_k(x).$}\label{mor2}
\end{figure}
\begin{ex}Define $h_k(x)=kx^2+\frac{1}{k}x+\frac{1}{k}.$ Then $$e_1h_k(x)=\frac{k}{2k+1}x^2+\frac{1}{k(2k+1)}x+\frac{4k^2+2k-1}{2k^2(2k+1)}.$$ Letting $k\rightarrow\infty,$ we have $h_k\overset{e}\to h$ with
\begin{align*}
h(x)&=\iota_{\{0\}}(x),\mbox{ and}\\
e_1h(x)&=\frac{1}{2}x^2.
\end{align*}\end{ex}
Figure \ref{mor3} shows the behaviour of the graphs as a function of $k.$
\begin{figure}[H]
\begin{center}\includegraphics[scale=0.4]{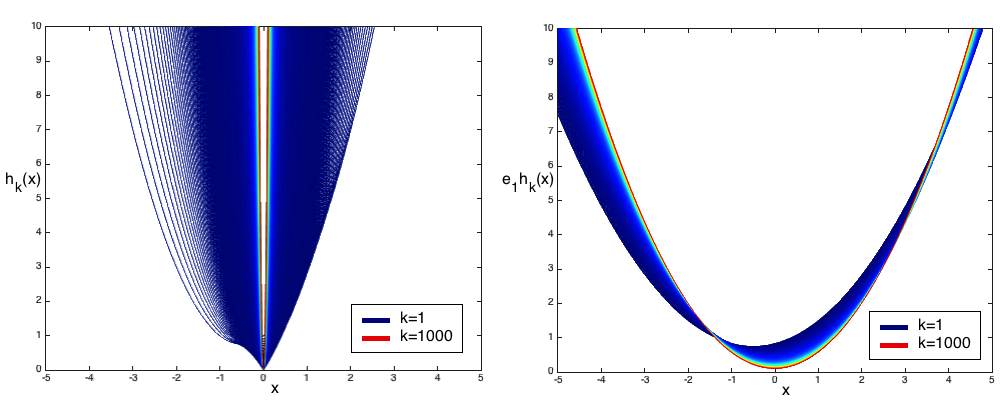}\end{center}
\caption{Left: $h_k(x).$ Right: $e_1h_k(x).$}\label{mor3}
\end{figure}

\section{Generalized linear-quadratic functions on $\R^n$}\label{sec:epi2}

Now we move on to finite-dimensional space. One natural goal that arises is that of unifying $f(x)=\frac{1}{2}\langle x,Ax\rangle+\langle b,x\rangle+c$ and $f(x)=\iota_{\{b\}}(x)+c$ in the more general setting of $\R^n.$ To do so, we first need to establish several properties of monotone linear relations and generalized linear-quadratic functions.

\subsection{Linear relations and generalized linear-quadratic functions}

\begin{df}
An operator $A:\R^n\rightrightarrows\R^n$ is a \emph{linear relation} if the graph of $A$ is a linear subspace of $\R^{n\times n}.$
\end{df}
\begin{ex}
On $\R,$ `monotone' is equivalent to `increasing'. So a monotone linear relation $A:\R\rightrightarrows\R$ must be a straight line with nonnegative slope, and since it is a subspace it must pass through the origin. There are three possibilities then: the $x$-axis, a line through the origin with positive slope, and the $y$-axis (see \cite[Theorem 12.15]{rockwets} for details):
\begin{itemize}
\item[(i)]$\gra A=\R\times\{0\}\Rightarrow A\equiv0,$
\item[(ii)]$\gra A=\Span\{(a,b)\},~a,b\in\R\setminus\{0\}\Rightarrow A(x)=kx,$ $k>0,$
\item[(iii)]$\gra A=\{0\}\times\R\Rightarrow A=N_{\{0\}}.$
\end{itemize}
\end{ex}
\begin{figure}[H]
\includegraphics[scale=0.35]{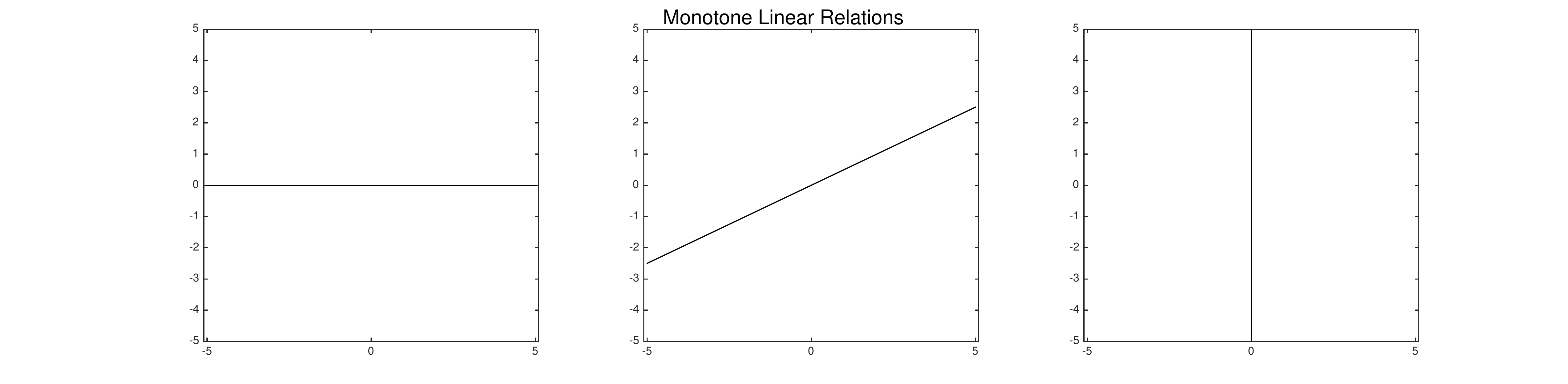}
\caption{The three possible forms of a monotone linear relation on $\R.$}
\end{figure}
\begin{df}
A \emph{generalized linear-quadratic function} $p:\R^n\rightarrow\OR$ is defined
$$p(x)=\frac{1}{2}\langle x-a,A(x-a)\rangle+\langle b,x\rangle+c\quad\forall x\in\R^n,$$where $A$ is a linear relation, $a,b\in\R^n,$ $c\in\R.$
\end{df}
\noindent Our first question is: why is the function $p$ well-defined?
\begin{ex}
Define$$A(x_1,x_2)=\{t(1,1):t\in\R\}\subseteq\R^2,~\forall(x_1,x_2)\in\R^2.$$ Then $A$ is a linear relation but not monotone, and $\langle x,Ax\rangle$ is not single-valued.
\end{ex}
\begin{proof}
It is elementary to show that $A$ is a linear relation. Let $x_1+x_2\neq0.$ Then
\begin{align*}
\langle(x_1,x_2),A(x_1,x_2)\rangle&=\{\langle(x_1,x_2),t(1,1)\rangle:t\in\R\}\\
&=\{t(x_1+x_2):t\in\R\}=\R.
\end{align*}Therefore, $\langle x,Ax\rangle$ is not single-valued. Observe that $A$ is not monotone. Indeed, set $t>0,$ and choose $x_1,x_2$
such that $x_1+x_2<0$ and $t(1,1)\in A(x_1,x_2).$ Note that $(0,0)\in A(0,0).$ Then
\begin{align*}
\langle(x_1,x_2)-(0,0),A[(x_1,x_2)-(0,0)]\rangle&=\langle(x_1,x_2),t(1,1)\rangle\\
&=t(x_1+x_2)<0.
\end{align*}
\end{proof}
\noindent The following fact says that when $A$ is a monotone linear relation, $p$ is well-defined.
\begin{fact}\emph{\cite[Proposition 3.2.1]{liangjin}}\label{liangprop}
Let $A:\R^n\rightrightarrows\R^n$ be a linear relation. Then $A$ is monotone if and only if $\langle x,Ax\rangle\geq0$ and $\langle x,Ax\rangle$ is single-valued for all $x\in\dom A.$
\end{fact}
\noindent Our next question is: why do we write $\langle x-a,A(x-a)\rangle$?
\begin{ex}
Consider the example on $\R$ of $A=N_{\{0\}}:$$$N_{\{0\}}(1-1)=\R\neq N_{\{0\}}(1)+N_{\{0\}}(-1)=\varnothing+\varnothing.$$
\end{ex}
\begin{fact}\emph{\cite[Proposition 3.1.3]{liangjin}}\label{liangjinfactlinear}
The operator $A$ is a linear relation if $~\forall\alpha,\beta\in\R$ and $~\forall x,y\in\R^n$ we have
$$A(\alpha x+\beta y)=\alpha Ax+\beta Ay+A0.$$
\end{fact}
\begin{prop}
Assume that $A$ is a monotone linear relation. If both $x,a\in\dom A$ or $\dom A=\R^n$, then
$$\langle x-a,A(x-a)\rangle=\langle x,Ax\rangle-\langle x,Aa\rangle-\langle a,Ax\rangle+\langle a,Aa\rangle.$$
\end{prop}
\begin{proof} When $A$ is a monotone linear relation $A0\subset \dom A^{\perp}$. When $x,a\in\dom A$,
we have that
$\langle x,Ax\rangle, \langle a,Aa\rangle, \langle x,Aa\rangle$ and $\langle a,Ax\rangle$ are single-valued. It suffices to
apply Fact~\ref{liangjinfactlinear}.
\end{proof}
\begin{df}
The \emph{adjoint} $A^*$ of  a linear relation $A$ is defined in terms of its graph:
\begin{align*}
\gra A^*&=\{(x^{**},x^*)\in\R^{n\times n}:(x^*,-x^{**})\in(\gra A)^\perp\}\\
&=\{(x^{**},x^*)\in\R^{n\times n}:\langle a,x^*\rangle=\langle a^*,x^{**}\rangle~\forall(a,a^*)\in\gra A\}.
\end{align*}
\end{df}
\begin{df}\label{def:symmetric}
An operator $A$ is \emph{symmetric} if $\gra A\subseteq\gra A^*,$ where $A^*$ is the adjoint of $A.$ Equivalently, $A$ is symmetric if $\langle x,y^*\rangle=\langle y,x^*\rangle~\forall (x,x^*),(y,y^*)\in\gra A.$
\end{df}
\begin{ex} The following are maximally monotone symmetric linear relations.
\begin{enumerate}
\item[(i)] A symmetric positive semidefinite matrix $A:\R^n\rightarrow\R^n$, and its set-valued
inverse
$A^{-1}:\R^n\rightarrow\R^n$. This follows from
$$(A^{-1})^*=(A^*)^{-1}=A^{-1}.$$
\item[(ii)] The normal cone operator $N_{L}:\R^n\rightrightarrows\R^n$, where $L\subset\R^n$ is a subspace. This is because
$$\gra N_{L}=L\times L^{\perp},\quad \gra (N_{L})^*=L\times L^{\perp}.$$
\end{enumerate}
\end{ex}
\begin{df}
For a monotone linear relation $A:\R^n\rightrightarrows\R^n,$ we define
\begin{itemize}
\item[(i)] $q_A(x)=\begin{cases}\frac{1}{2}\langle x,Ax\rangle,& \mbox{ if }x\in\dom A,\\
\infty,& \mbox{ if }x\not\in\dom A,\end{cases}$
\item[(ii)] $A_+=\frac{1}{2}(A+A^*).$
\end{itemize}
\end{df}
\begin{rem}\label{normalrem}
The framework of a generalized linear-quadratic function is more convenient. For instance, for $a\in\R^n$ and $c\in\R$ the indicator function $f:\R^n\rightarrow\OR,$
$$f(x)=\iota_{\{a\}}(x)+c=\begin{cases}
c,&\mbox{if }x=a,\\\infty,&\mbox{if }x\neq a,
\end{cases}$$can be expressed as a generalized linear-quadratic function:
$$f(x)=q_{N_{\{0\}}}(x-a)+c,$$
where $$N_{\{0\}}(x)=\begin{cases}\R^n,&\mbox{if }x=0,\\\varnothing,&\mbox{if }x\neq0\end{cases}$$ is a maximally monotone linear relation.\end{rem}
\begin{prop}\label{prop:symormono}
Let $A:\R^n\rightrightarrows\R^n$ be a linear relation. Suppose that either
\begin{itemize}
\item[(i)]$A$ is symmetric, or
\item[(ii)]$A$ is monotone.
\end{itemize}
Then $q_A$ is an extended-real-valued function.
\end{prop}
\begin{proof}
(i) Let $A$ be symmetric. Then by Definition \ref{def:symmetric} with $y=x,$ we have$$\langle x,y^*\rangle=\langle x,x^*\rangle~\forall(x,x^*),(x,y^*)\in\gra A.$$ That is, $q_A(x)=\langle x,Ax\rangle=\langle x,x^*\rangle$ is single-valued for all $x\in\dom A.$\\

\noindent (ii) This is direct from Fact \ref{liangprop}.
\end{proof}

\subsection{Properties and calculus of $q_A$}

The generalized linear-quadratic function $q_A$ is instrumental in establishing our final main result. In this section, we collect a number of properties of $q_A$ under conditions such as maximal monotonicity and symmetry.
\begin{lem}\label{lem:invsym}
Let $A:\R^n\rightrightarrows\R^n$ be symmetric. Then $A^{-1}$ is symmetric.
\end{lem}
\begin{proof}
By definition, $A$ is symmetric if and only if\begin{equation}\label{uva}\langle x,Ay\rangle=\langle Ax,y\rangle\qquad\forall x,y\in\dom A.\end{equation} Let $u\in Ay,$ $v\in Ax.$ Then $u,v\in\ran A=\dom A^{-1},$ and $x\in A^{-1}v,$ $y\in A^{-1}u.$ Substituting into \eqref{uva}, we have
$$\langle A^{-1}v,u\rangle=\langle v,A^{-1}u\rangle\qquad\forall u,v\in\dom A^{-1},$$which is the definition of symmetry of $A^{-1}.$
\end{proof}
\begin{lem}\label{lem:sumsym}
Let $A_1,A_2:\R^n\rightrightarrows\R^n$ be maximally monotone linear relations. Then $A_1+A_2$ is a maximally monotone linear relation. If, in addition, $A_1$ and $A_2$ are symmetric, then $A_1+A_2$ is symmetric.
\end{lem}
\begin{proof}
Since $\dom A_1$ and $\dom A_2$ are linear subspaces of $\R^n,$ $\dom A_1-\dom A_2$ is a closed subspace. By \cite[Theorem 7.2.2]{liangjin}, $A_1+A_2$ is maximally monotone. Since $\gra A_1$ and $\gra A_2$ are linear subspaces, $\gra(A_1+A_2)$ is a linear subspace. Hence, $A_1+A_2$ is a linear relation. It remains to prove that $A_1+A_2$ is symmetric. Let $(x,x^*),(y,y^*)\in\gra(A_1+A_2)$ be arbitrary. Since $\dom(A_1+A_2)=\dom A_1\cap\dom A_2,$ we have $x,y\in\dom A_1$ and $x,y\in\dom A_2.$ Then there exist $x_1^*,y_1^*\in\ran A_1$ and $x_2^*,y_2^*\in\ran A_2$ such that
\begin{itemize}
\item[(i)]$(x,x_1^*),(y,y_1^*)\in\gra A_1$ and $(x,x_2^*),(y,y_2^*)\in\gra A_2,$ and
\item[(ii)]$x_1^*+x_2^*=x^*$ and $y_1^*+y_2^*=y^*.$
\end{itemize} This gives us that
$$(x,x_1^*+x_2^*)=(x,x^*)\in\gra(A_1+A_2)\mbox{ and }(y,y_1^*+y_2^*)=(y,y^*)\in\gra(A_1+A_2).$$
Now consider $\langle x,y^*\rangle-\langle y,x^*\rangle:$
\begin{align*}
\langle x,y^*\rangle-\langle y,x^*\rangle&=\langle x,y_1^*\rangle+\langle x,y_2^*\rangle-\langle y,x_1^*\rangle-\langle y,x_2^*\rangle\\
&=(\langle x,y_1^*\rangle-\langle y,x_1^*\rangle)+(\langle x,y_2^*\rangle-\langle y,x_2^*\rangle)\\
&=(\langle x,y_1^*\rangle-\langle x,y_1^*\rangle)+(\langle x,y_2^*\rangle-\langle x,y_2^*\rangle)\\
&\mbox{\qquad\footnotesize($A_1$ is symmetric)\qquad\quad($A_2$ is symmetric)\normalsize}\\
&=0.
\end{align*}
Thus, for any arbitrary $(x,x^*),(y,y^*)\in\gra(A_1+A_2)$ we have that $\langle x,y^*\rangle=\langle y,x^*\rangle.$ Therefore, $A_1+A_2$ is symmetric.
\end{proof}
\begin{prop} Let $A_1,A_2$ be maximally monotone symmetric linear relations on $\R^n.$ Then $A_1^*+A_2^*=(A_1+A_2)^*.$
\end{prop}
\begin{proof}$(\Rightarrow)$ By definition of symmetry, we have \begin{equation}\label{symeq}\gra A_1\subseteq\gra A_1^*.\end{equation} Since $A_1$ is maximally monotone, $A_1^*$ is also maximally monotone by \cite[Corollary 5.11]{bauschke2012brezis}. Then \eqref{symeq} is actually an equality and we have
\begin{equation}\label{eq:sstar}
A_1=A_1^*,\mbox{ and similarly }A_2=A_2^*.\end{equation}
Then by definition of adjoint,
\begin{align*}
\gra((A_1+A_2)^*)&=\{(x,x^*)\in\R^{n\times n}:(x^*,-x)\in(\gra(A_1+A_2))^\perp\}\\
&=\{(x,x^*)\in\R^{n\times n}:(x^*,-x)\in(\gra(A_1^*+A_2^*))^\perp\}\qquad\mbox{by \eqref{eq:sstar}}\\
&=\gra((A_1^*+A_2^*)^*).
\end{align*}
Once more by definition of symmetry, we have $\gra(A_1^*+A_2^*)\subseteq\gra((A_1^*+A_2^*)^*).$ Therefore, $\gra(A_1^*+A_2^*)\subseteq\gra((A_1+A_2)^*).$\\

\noindent$(\Leftarrow$) We have $\gra((A_1+A_2)^*)=\gra((A_1^*+A_2^*)^*)$ from above, and by symmetry $\gra((A_1^*+A_2^*)^*)\subseteq\gra((A_1^*+A_2^*)^{**}).$ Since we are in $\R^n,$ $\gra((A_1^*+A_2^*)^{**})$ is closed and thus $\gra((A_1^*+A_2^*)^{**})=\gra(A_1^*+A_2^*).$ Therefore, $\gra((A_1+A_2)^*)\subseteq\gra(A_1^*+A_2^*).$\footnote{Thank you to Dr. Walaa Moursi for contributing to this proof.}
\end{proof}
\begin{prop}\label{fourprops}
Let $A$ be a maximally monotone linear relation. Then
\begin{itemize}
\item[(i)] $q_A$ is well-defined, i.e. $q_A:\R^n\rightarrow\R\cup\{\infty\},$
\item[(ii)] $q_A$ is convex,
\item[(iii)] $q_A=q_{A_+},$
\item[(iv)] $\partial q_A=A_+.$
\end{itemize}
\end{prop}
\begin{proof}
(i) This is direct from Proposition \ref{prop:symormono}.\\

\noindent(ii) See \cite[Proposition 2.3]{monolinrel}.\\

\noindent(iii) Since $A$ is maximally monotone, $A^*$ is also maximally monotone \cite[Corollary 5.11]{bauschke2012brezis}. By definition of $A^*,$ we have $\langle x,Ax\rangle=\langle A^*x,x\rangle=\langle x,A^*x\rangle.$ Then
\begin{align*}
q_A(x)=\frac{1}{2}\langle x,Ax\rangle&=\frac{1}{2}\left(\frac{\langle x,Ax\rangle+\langle x,A^*x\rangle}{2}\right)\\
&=\frac{1}{2}\left\langle x,\frac{Ax+A^*x}{2}\right\rangle\\
&=\frac{1}{2}\langle x,A_+x\rangle=q_{A_+}(x).
\end{align*}
(iv) Since $A$ is maximally monotone, $A^*$ is as well, hence $A_+$ is as well. Then
$$\partial q_A=\partial q_{A_+}=A_+=\frac{1}{2}(A+A^*).$$
\end{proof}
\begin{lem}\label{symlemma}
Let $A$ be a maximally monotone symmetric linear relation. Then $\partial q_A=A.$
\end{lem}
\begin{proof}
Since $A$ is symmetric, $A=A^*.$ The result follows from Proposition \ref{fourprops}(iv).\end{proof}
\begin{cor}\label{cora1a2}
Let $A_1,$ $A_2$ be maximally monotone symmetric linear relations such that $q_{A_1}=q_{A_2}.$ Then $A_1=A_2.$
\end{cor}
\begin{proof}This follows from $\partial q_{A_1}=A_1,$ $\partial q_{A_2}=A_2.$ \end{proof}
\begin{rem}
The maximal monotonicity condition of Corollary \ref{cora1a2} is necessary. As a counterexample, consider a monotone selection $S$ of $A$ and set
$$A_1=S,\qquad A_2=S+A0.$$Then $q_{A_1}=q_{A_2},$ but $A_1\neq A_2$ unless $A0=\{0\}.$
\end{rem}
\begin{prop}\label{p:difference}
Let $A_1,A_2$ be monotone linear relations. Then$$q_{A_1}+q_{A_2}=q_{A_1+A_2}.$$In addition, if $\dom A_1\subseteq\dom A_2$ and $A_1-A_2$ is monotone, then$$q_{A_1}-q_{A_2}=q_{A_1-A_2}.$$
\end{prop}
\begin{proof}
By definition, we have $q_{A_1}(x)=\frac{1}{2}\langle x,A_1x\rangle$ if $x\in\dom A_1,$ $\infty$ otherwise. Similarly, $q_{A_2}(x)=\frac{1}{2}\langle x,A_2x\rangle$ if $x\in\dom A_2,$ $\infty$ otherwise. Thus,
\begin{align*}(q_{A_1}+q_{A_2})(x)&=\begin{cases}
\frac{1}{2}\langle x,(A_1+A_2)x\rangle,&\mbox{if }x\in\dom A_1\cap\dom A_2,\\
\infty,&\mbox{else.}
\end{cases}\\
&=q_{A_1+A_2}(x).\end{align*}
Now suppose that $\dom A_1\subseteq\dom A_2$ and $A_1-A_2$ is monotone. Then for $x\in\dom A_2$ with $x\in\dom A_1,$ we have that $q_{A_1}-q_{A_2}$ is single-valued, so that$$q_{A_1}(x)-q_{A_2}(x)=q_{A_1-A_2}(x).$$When $x\not\in\dom A_1,$ we have$$q_{A_1}(x)-q_{A_2}(x)=\infty-q_{A_2}(x)=\infty.$$Now$$\dom q_{A_1-A_2}=\dom(A_1-A_2)=\dom A_1\cap\dom A_2=\dom A_1,$$so that$$q_{A_1-A_2}(x)=\infty\mbox{ when }x\not\in\dom A_1.$$
Therefore, $$q_{A_1}-q_{A_2}=q_{A_1-A_2}.$$
\end{proof}
The condition $\dom A_1\subseteq\dom A_2$ is necessary for $q_{A_1}-q_{A_2}=q_{A_1-A_2}.$ The following example shows that Proposition~\ref{p:difference} fails if
$\dom A_{1}\not\subseteq \dom A_{2}$.
\begin{ex} Let $A_{1}, A_{2}:\R^2\rightrightarrows\R^2$ be maximally monotone linear relations given by
$$A_{1}=\Id, \quad A_{2}=N_{\R\times \{0\}}$$ where
$$N_{\R\times\{0\}}(x,y)=\begin{cases}
\{0\}\times\R &\mbox{if }y=0,\\
\varnothing &\mbox{if }y\neq 0.
\end{cases}
$$
Then
$$(A_{1}-A_{2})(x,y)=\begin{cases}
(x,0)-\{0\}\times\R &\mbox{if }y=0,\\
\varnothing & \mbox{if }y\neq 0
\end{cases}
$$
is a maximally monotone linear relation.
We have
$$q_{A_{1}}(0,1)-q_{A_{2}}(0,1)=1/2-\infty=-\infty,$$
but
$$q_{A_{1}-A_{2}}(0,1)=\infty$$
because $(0,1)\not\in \dom(A_{1}-A_{2}).$
\end{ex}
\begin{prop}\label{tfae5}
Let $A$ be a maximally monotone symmetric linear relation. Then the following are equivalent:
\begin{itemize}
\item[(i)]$q_A(x)=0;$
\item[(ii)]$x\in\argmin q_A;$
\item[(iii)]$0\in\partial q_A(x);$
\item[(iv)]$0\in Ax;$
\item[(v)]$x\in A^{-1}0.$
\end{itemize}
\end{prop}
\begin{proof}(i)$\Rightarrow$(ii) Let $q_A(x)=0.$ Since $A$ is monotone, by Fact \ref{liangprop} we have $q_A(y)\geq0~\forall y\in\R^n.$ Hence,
$$\min\limits_{y\in\R^n}q_A(y)=0=q_A(x)\Rightarrow x\in\argmin q_A.$$\\
(ii)$\Rightarrow$(iii) This is direct from Fermat's Theorem, every local extremum of $q_A$ is a stationary point.\\

\noindent(iii)$\Rightarrow$(iv) Let $0\in\partial q_A(x).$ Since $A$ is symmetric and maximally monotone, by Lemma \ref{symlemma} we have $\partial q_A(x)=A(x).$ Therefore, $0\in Ax.$\\

\noindent(iv)$\Rightarrow$(i) Let $0\in Ax.$ Then, since $q_A(x)$ is single-valued by Fact \ref{liangprop}, we have$$q_A(x)=\frac{1}{2}\langle x,Ax\rangle=\frac{1}{2}\langle x,0\rangle=0.$$
\noindent(iv)$\Leftrightarrow$(v) We have
$0\in Ax\Leftrightarrow x\in A^{-1}0.$
\end{proof}

\subsection{The Fenchel conjugate of $q_A$}

Conjugacy plays a vital role in convex analysis \cite[Chapter X]{hiriart2013convex}. One often finds it quite beneficial to work temporarily in a dual space in order to solve a problem, then return the answer to the primal space. In this section, we explore the Fenchel conjugate of $q_A.$ We show that the set-valued inverse $A^{-1}$ is more convenient for computing $q_A^*.$
\begin{prop}\label{qastartheorem}
Let $A:\R^n\rightrightarrows\R^n$ be a maximally monotone symmetric linear relation. Then $q_A^*=q_{A^{-1}},$ that is,
$$q_A^*(y)=\begin{cases}\frac{1}{2}\langle y,A^{-1}y\rangle,&\mbox{if }y\in\ran A,\\\infty,&\mbox{if }y\not\in\ran A.\end{cases}$$
Consequently,
$$q_A^{**}(x)=\begin{cases}\frac{1}{2}\langle x,Ax\rangle,&\mbox{if }x\in\ran A^{-1},\\\infty,&\mbox{if }x\not\in\ran A^{-1}.\end{cases}$$
Thus, $q_A^{**}=q_A,$ so $q_A$ is lsc and convex.
\end{prop}
\begin{proof}
By the definition of $q_A^*,$ we have
\begin{align}
q_A^*(y)&=\sup\limits_x\{\langle y,x\rangle-q_A(x)\}\nonumber\\
&=\sup\limits_{x\in\dom A}\left\{\langle y,x\rangle-\frac{1}{2}\langle x,Ax\rangle\right\}.\label{supthm}
\end{align}
Consider two cases.\\

\noindent(i) Let $y\in\ran A.$ Then the solution to the supremum in \eqref{supthm} is $\bar{x}$ such that $0\in\partial\left[\langle y,\bar{x}\rangle-\frac{1}{2}\langle\bar{x},A\bar{x}\rangle\right].$ This gives $y\in A\bar{x},$ hence, $\bar{x}\in A^{-1}y.$ Then
$$q_A^*(y)=\langle y,A^{-1}y\rangle-\frac{1}{2}\langle A^{-1}y,y\rangle=\frac{1}{2}\langle y,A^{-1}y\rangle.$$
(ii) Let $y\not\in\ran A.$ Note that since $\ran A$ is closed and convex, by \cite[Corollary 11.4.2]{convanalrock} there exist $z\in\R^n$ and $r\in\R$ such that
$$\langle z,y\rangle>r\geq\sup\limits_{x\in\ran A}\langle x,z\rangle.$$ Since $\ran A$ is a subspace, we have $0\in\ran A$ and $r\geq0.$ Also since $\ran A$ is a subspace, we have $kx\in\ran A~\forall x\in\ran A,~\forall k\in\R.$ Hence, $r\geq k\langle x,z\rangle$ for all $x\in\ran A$ and for all $k\in\R.$ Thus, $\langle x,z\rangle=0$ for all $x\in\ran A$ (otherwise, for $\langle x,z\rangle\neq0$ one could choose $k$ such that $k\langle x,z\rangle>r$). Then $\sup_{x\in\ran A}\langle x,z\rangle=0,$ hence $\langle z,y\rangle>0.$ Now noting that
\begin{equation}\label{supthm2}
\sup\limits_{k>0}\left\{\langle y,kz\rangle-\frac{1}{2}\langle kz,A(kz)\rangle\right\}=\sup\limits_{k>0}\langle y,kz\rangle=\infty,\end{equation} and that the supremum of \eqref{supthm} is greater than or equal to that of \eqref{supthm2}, we have $q^*_A(y)=\infty.$
\end{proof}
\begin{cor}
Let $A$ be symmetric, positive definite and nonsingular. Then $q_A^*=q_{A^{-1}},$ where $A^{-1}$ is the classical inverse.
\end{cor}
\begin{cor}\label{qcor2}
Let $A$ be symmetric and positive semidefinite. Then $q_A^*=q_{A^{-1}}$ where
$$q_{A^{-1}}(x)=\begin{cases}\frac{1}{2}\langle x,A^{-1}x\rangle,&\mbox{if }x\in\ran A,\\\infty,&\mbox{if }x\not\in\ran A,\end{cases}$$ and $A^{-1}$ is the set-valued inverse.
\end{cor}
\begin{prop}
Let $A$ be a maximally monotone symmetric linear relation. Then $q_A^*=q_A$ if and only if $A=A^{-1},$ if and only if $A=\Id.$
\end{prop}
\begin{proof}
The proof that $A=A^{-1}$ if and only if $A=\Id$ is found in \cite[Proposition 2.8]{resaverage}. In the sequel, we prove that $q_A^*=q_A$ if and only if $A=A^{-1}.$\\
$(\Leftarrow)$ Suppose that $A=A^{-1}.$ Then we see immediately by Corollary \ref{qcor2} that $q_A^*=q_A.$\\
$(\Rightarrow)$ Suppose that $q_A^*=q_A.$ Then by Lemma \ref{symlemma}, we have that $\partial q_A=A=\partial q_A^*.$ By Corollary \ref{qcor2}, we find $\partial q_A^*=A^{-1}.$ Therefore, $A=A^{-1}.$
\end{proof}
\begin{prop}\label{parallelsum}
Let $A_i:\R^n\rightrightarrows\R^n$ be a maximally monotone symmetric linear relation for each $i\in\{1,\ldots,m\}.$ Then the infimal convolution $f=q_{A_1}\oblong\cdots\oblong q_{A_m}$ is a generalized linear-quadratic function and
$$\partial f=\left(\sum\limits_{i=1}^nA^{-1}_i\right)^{-1},$$which is the \emph{parallel sum} of $A_i.$
\end{prop}
\begin{proof}
By Lemmas \ref{lem:invsym} and \ref{lem:sumsym}, we have that $A_1^{-1}+\cdots+A_m^{-1}$ is a maximally monotone symmetric linear relation. By Corollary \ref{qcor2}, $q^*_{A_i}=q_{A_i^{-1}}.$ Since $0\in\bigcap_{i=1}^m\ri\ran A_i,$ \cite[Theorem 16.4]{convanalrock} gives
\begin{align*}
q_{(A_1^{-1}+\cdots+A_m^{-1})^{-1}}=q^*_{A_1^{-1}+\cdots+A_m^{-1}}&=\left(q_{A_1^{-1}}+\cdots+q_{A_m^{-1}}\right)^*\\
&=\left(q^*_{A_1}+\cdots+q^*_{A_m}\right)^*\\
&=q_{A_1}\oblong\cdots\oblong q_{A_m}=f.
\end{align*}Therefore, by Lemma \ref{symlemma}, we have the statement of the proposition.
\end{proof}
\begin{prop}
Let $A_1:\R^n\rightrightarrows\R^n$ be a maximally monotone symmetric linear relation and $A_2:\R^n\rightarrow\R^n$ be symmetric positive definite. Define $h:\R^n\rightarrow\OR$ by
\begin{equation}\label{stardif}
q_{A_2}\oblong h=q_{A_1}.\end{equation} Then for every $x\in\R^n,$
\begin{equation}\label{stardif2}
h(x)=\left(q^*_{A_1}-q^*_{A_2}\right)^*(x)=\sup\limits_y\left[q_{A_1}(x+y)-q_{A_2}(y)\right].\end{equation} Consequently, when $A_1^{-1}-A_2^{-1}$ is monotone, one has
\begin{equation}\label{stardif3}
\partial h=\left(A_1^{-1}-A_2^{-1}\right)^{-1},
\end{equation}which is the \emph{star-difference} of $A_1$ and $A_2.$
\end{prop}
\begin{proof}
Taking the Fenchel conjugate of \eqref{stardif} yields $h^*=q_{A_1}^*-q_{A_2}^*.$ Then by Toland-Singer duality, we have \eqref{stardif2}. Observe that $A_1^{-1}-A_2^{-1}$ is maximally monotone because of the following. We have $$\dom(A_1^{-1}-A_2^{-1})=\ran A_1\cap\ran A_2=\ran A_1=\dom A_1^{-1},\mbox{ and}$$
$$(A_1^{-1}-A_2^{-1})(0)=A_1^{-1}(0).$$ Because $A_1^{-1}$ is maximally monotone, $(\dom A_1^{-1})^\perp=A_1^{-1}(0).$ Then by \cite[Fact 2.4(v)]{bauschke2010borwein}, $A_1^{-1}-A_2^{-1}$ is maximally monotone. Since $q^*_{A_i}=q_{A_i^{-1}},$ $q_{A_1^{-1}}-q_{A_2^{-1}}=q_{A_1^{-1}-A_2^{-1}}$ and $A_1^{-1}-A_2^{-1}$ is a maximally monotone symmetric linear relation, we have \eqref{stardif3}.
\end{proof}
\begin{rem}
This result generalizes that of Hiriart-Urruty \cite{hiriart1994deconvolution}, because $A_1^{-1}-A_2^{-1}$ need not be positive definite. (See \cite[Example 2.7]{hiriart1994deconvolution}.)
\end{rem}

\subsection{Relating the set-valued inverse and the Moore--Penrose inverse}

The set-valued inverse $A^{-1}$ of a linear mapping and the Moore--Penrose inverse $A^\dagger$ both have their uses. For properties of $A^\dagger,$ see \cite[p. 423--428]{matanalapp}. In this section, we show how the two inverses are closely related. We also include a description of the Moore--Penrose inverse for a particular mapping, the orthogonal projector.
\begin{prop}\label{threeprops}
The following hold.
\begin{itemize}
\item[(i)] Let $A:\R^n\rightrightarrows\R^n$ be a linear mapping. Then
$$A^{-1}x=\begin{cases}
A^\dagger x+A^{-1}0,&\mbox{ if }x\in\ran A,\\
\varnothing,&\mbox{ if }x\not\in\ran A.
\end{cases}$$
\item[(ii)]\label{i:two} Let $A:\R^n\rightrightarrows\R^n$ be maximally monotone. Then
$$A^{-1}x=A^\dagger x+N_{\dom A^{-1}}=\begin{cases}
A^\dagger x+(\ran A)^\perp,&\mbox{ if }x\in\ran A,\\
\varnothing,&\mbox{ if }x\not\in\ran A.
\end{cases}$$
\item[(iii)] Let $A:\R^n\rightrightarrows\R^n$ be monotone, symmetric and linear. Then
$$A^{-1}=P_{\ran A}A^\dagger P_{\ran A}+N_{\dom A^{-1}}.$$
\end{itemize}
\end{prop}
\begin{proof}
(i) Since $AA^\dagger=P_{\ran A}~\forall x\in\ran A,$ it holds that
$$AA^\dagger x=P_{\ran A}x=x\Rightarrow A^\dagger x\in A^{-1}x.$$ Since $A^{-1}x=x^*+A^{-1}0$ for every $x^*\in Ax,$ we have
$$A^{-1}=A^\dagger+A^{-1}0\mbox{ on }\ran A.$$
(ii) Since $A$ is maximally monotone, $A^{-1}$ is as well, and $$(\dom A^{-1})^\perp=A^{-1}0.$$ Applying part (i) completes the proof.\\

\noindent(iii) If $A$ is maximally monotone and linear, then
$$\ran A^\dagger=\ran A^\top=\ran A^*=\ran A,\mbox{ and}$$
$$N_{\dom A^{-1}}(x)=(\dom A^{-1})^\perp=A^{-1}0.$$ This implies that on $\ran A=\dom A^{-1},$
$$P_{\ran A}A^\dagger P_{\ran A}=A^\dagger.$$ Now we apply part (ii). Let $x\in\R^n,$ $u=P_{\ran A}x.$ Denote $P_{\ran A}A^\dagger P_{\ran A}$ by $L.$ Using $AA^\dagger A=A,$ $L$ is monotone because
\begin{align*}
\langle x,Lx\rangle&=\langle P_{\ran A}x,A^\dagger P_{\ran A}x\rangle\\
&=\langle Au,A^\dagger Au\rangle\\
&=\langle u,AA^\dagger Au\rangle\\
&=\langle u,Au\rangle\\&\geq0.
\end{align*}
We have that $L$ is symmetric, because
$$(A^\dagger)^*=(A^*)^\dagger=A^\dagger.$$
\end{proof}
In \cite[Exercise 3.13]{convmono}, for a linear mapping $A$, one has
\begin{equation}\label{e:heinz:comb}
A^{\dag}=P_{\ran A^*}A^{-1}P_{\ran A}.
\end{equation}
For a set $\Omega\subset\R^n$, define the indicator mapping $\vartriangle:\R^n\rightarrow\R^n$ of $\Omega$ relative to
$\R^n$ by
$$\vartriangle_{\Omega}(x)=\begin{cases}
0\in\R^n &\text{ if $x\in\Omega$,}\\
\varnothing &\text{ if $x\not\in\Omega$.}
\end{cases}$$
(See, e.g., \cite{mordukhovich2006variational}.) Combing \eqref{e:heinz:comb} and Proposition~\ref{threeprops}, we obtain a
complete relationship between $A^{-1}$ and $A^{\dag}$.
\begin{cor}\label{cor:linproj}
\begin{enumerate}
\item[(i)] When $A$ is  linear mapping on $\R^n$,
$$A^{-1}=
A^{\dag}+\vartriangle_{\dom A^{-1}}+A^{-1}0, \text{ and }
$$
$$A^{\dag}=P_{\ran A^*}A^{-1}P_{\ran A}.$$
\item[(ii)] If, in addition, $A$ is maximally monotone, then
$$A^{-1}=A^{\dag}+N_{\dom A^{-1}},$$
and
$$A^{\dag}=P_{\ran A}A^{-1}P_{\ran A}.$$
\end{enumerate}
\end{cor}
\noindent Corollary \ref{cor:linproj}(i) is a corollary of Proposition \ref{threeprops}.
\begin{cor}\label{c:conj:dag}
Let $A$ be a maximally monotone symmetric linear relation. Then
$$(q_A)^*=\begin{cases}
q_{A^\dagger},&\mbox{if }x\in\ran A,\\\infty,&\mbox{if }x\not\in\ran A.
\end{cases}$$
\end{cor}
In the sequel we present the Moore--Penrose inverse of the projector mapping. We remind the reader of the definition.
\begin{df}
Let $C\subseteq\R^n$ be closed and convex. The \emph{projection} of a point $x$ onto $C$ is defined
$$P_Cx=\{y\in C:\|y-x\|=d_C(x)\},$$
where $d_C$ is the \emph{distance function:} $d_C(x)=\inf\limits_{y\in C}\|y-x\|.$ We call $P_C$ the projection operator.
\end{df}
\begin{prop}\label{projectorthm}
Let $P_L$ be the orthogonal projector onto a subspace $L\subseteq\R^n.$ Then the Moore--Penrose generalized inverse $P^\dagger_L=P_L.$
\end{prop}
\begin{proof}
By \cite[Theorem 10.5]{roman2005advanced}, $P_L$ is idempotent $(P_L^2=P_L)$ and Hermitian $(P_L^*=P_L).$ Since $P^\dagger_L$ is the unique operator that satisfies the four Moore--Penrose equations, it is a simple matter to verify that each of them is satisfied by $P^\dagger_L=P_L:$
\begin{itemize}
\item[(i)] $AA^\dagger A=A:$ $P_LP_LP_L=P_L,$
\item[(ii)] $A^\dagger AA^\dagger=A^\dagger:$ $P_LP_LP_L=P_L,$
\item[(iii)] $(AA^\dagger)^*=AA^\dagger:$ $(P_LP_L)^*=P_LP_L\Rightarrow P_L^*=P_L,$
\item[(iv)] $(A^\dagger A)^*=A^\dagger A:$ $(P_LP_L)^*=P_LP_L\Rightarrow P_L^*=P_L.$
\end{itemize}
Therefore, $P_L^\dagger=P_L.$
\end{proof}
\begin{cor}
Let $f:\R^n\rightarrow\OR:$ $x\mapsto\frac{1}{2}\langle x,P_Lx\rangle,$ where $L$ is a subspace. Then
$$f^*(x^*)=\begin{cases}
\frac{1}{2}\langle x^*,P_Lx^*\rangle,&\mbox{ if }x^*\in L,\\\infty,&\mbox{ if }x^*\not\in L
\end{cases}=\begin{cases}\frac{1}{2}\langle x^*,x^*\rangle,&\mbox{if }x^*\in L,\\\infty,&\mbox{if }x^*\not\in L.\end{cases}$$
\end{cor}
\begin{proof}
Combining Proposition \ref{projectorthm} and Corollary~\ref{c:conj:dag}, the proof is immediate.
\end{proof}
\noindent We end this section with an extension of Rockafellar's and Wets' result \cite[Example 11.10]{rockwets}.
\begin{ex}\label{rockext} Let $A:\RR^n\rightrightarrows\RR^n$ be a symmetric monotone linear relation, $b\in\R^n$, $c\in \R$.
Suppose
$$f(x)=q_{A}(x)+\langle b,x\rangle+c.$$
Then for every $y\in\RR^n$, the Fenchel conjugate of $f$ is
$$f^*(y)=q_{A^{-1}}(y-b)-c=\begin{cases}
\frac{1}{2}\langle y-b,A^{\dag}(y-b)\rangle-c, &\mbox{if }y-b\in\ran A,\\
\infty, &\mbox{if }y-b\not\in\ran A.
\end{cases}
$$
\end{ex}
\begin{proof} Applying Theorem~\ref{qastartheorem}, we have $\forall y\in \R^n$,
\begin{align*}
f^*(y) &= (q_{A})^{*}(y-b)-c =q_{A^{-1}}(y-b)-c.
\end{align*}
By Proposition~\ref{threeprops}(ii),
$$A^{-1}=A^{\dag}+N_{\ran A}.$$
This gives
$$f^*(y)=\begin{cases}
\frac{1}{2}\langle y-b,A^{\dag}(y-b)\rangle-c, &\mbox{if }y-b\in\ran A,\\
\infty, &\mbox{if }y-b\not\in\ran A.
\end{cases}
$$
\end{proof}

\begin{rem}
In \cite[Example 11.10]{rockwets}, the authors assume that
$A\in \R^{n\times n}$, i.e., $A$ is a linear operator. In Example \ref{rockext}, $A$ is a linear relation.
\end{rem}

\subsection{Characterizations of Moreau envelopes}

In this section, we present several useful properties of Moreau envelopes of convex functions. We identify the form of the Moreau envelope for quadratic functions, and provide a characterization of Moreau envelopes that involves Lipschitz continuity. This leads to a sum rule for Moreau envelopes of convex functions. Then we follow up with one of the main results of this paper; Theorem \ref{nonthm} is a characterization relating generalized linear-quadratic functions to nonexpansive mappings.
\begin{prop}\label{1lip}
Let $f\in\Gamma_0(\R^n).$ Then $f=e_rg$ for some $g\in\Gamma_0(\R^n)$ if and only if $\nabla f$ is $r$-Lipschitz.
\end{prop}
\begin{proof}
$(\Rightarrow)$ Let $f=e_rg$ for some $g\in\Gamma_0(\R^n).$ Then by Proposition \ref{nablaprop}(ii) we have $\nabla f=r(\Id-\Prox_g^r).$ Let $x,y\in\R^n.$ Then
\begin{align*}
\|\nabla f(x)-\nabla f(y)\|&=r\|x-\Prox_g^r(x)-y+\Prox_g^r(y)\|&&\\
&=r\left\|\frac{1}{r}\Prox_{g^*}^{\frac{1}{r}}(rx)-\frac{1}{r}\Prox_{g^*}^{\frac{1}{r}}(ry)\right\|&&\mbox{(Proposition \ref{nablaprop}(v))}\\
&=\left\|\Prox_{g^*}^{\frac{1}{r}}(rx)-\Prox_{g^*}^{\frac{1}{r}}(ry)\right\|\\
&=\|J_{r\partial g^*}(rx)-J_{r\partial g^*}(ry)\|&&\mbox{(Fact \ref{resolventfact})}\\
&\leq\|rx-ry\|=r\|x-y\|.
\end{align*}
Therefore, $\nabla f$ is $r$-Lipschitz.\\
$(\Leftarrow)$ Let $\nabla f$ be $r$-Lipschitz. Then by Fact \ref{baillon}, $\frac{1}{r}\nabla f$ is firmly nonexpansive. By Fact \ref{nonexpequivmono}, $\frac{1}{r}\nabla f=(\Id+A)^{-1}$ for some $A$ monotone. Since $f\in\Gamma_0(\R^n),$ $A$ is in fact maximally cyclically monotone by Fact \ref{cyclic}. Thus, $A=\partial g$ for some $g\in\Gamma_0(\R^n)$ by Fact \ref{maxcyc}. Hence,
$$\nabla f=r(\Id+\partial g)^{-1}=\left[(\Id+\partial g)\circ\left(\frac{\Id}{r}\right)\right]^{-1}.$$
Then we have
\begin{align*}
\partial f^*=(\nabla f)^{-1}&=(\Id+\partial g)\left(\frac{\Id}{r}\right)\\
&=\frac{\Id}{r}+\partial g\left(\frac{\Id}{r}\right),
\end{align*}
so that $$f^*=\frac{q}{r}+rg\left(\frac{\cdot}{r}\right)+c,~c\in\R.$$ Taking the conjugate of both sides yields
\begin{align*}
f&=\left[\frac{q}{r}+rg\left(\frac{\cdot}{r}\right)+c\right]^*\\
&=\left(\frac{q}{r}\right)^*\oblong\left[rg\left(\frac{\cdot}{r}\right)+c\right]^*\\
&=(rq)\oblong(rg^*-c)\\
&=e_r(rg^*-c),
\end{align*}
where $g^*\in\Gamma_0(\R^n).$
\end{proof}
\begin{cor}\label{r1r2}
Let $r_1,r_2>0,$ $g,h\in\Gamma_0(\R^n).$ Then
\begin{equation}\label{r1plusr2}
e_{r_1}g+e_{r_2}h=e_{r_1+r_2}f
\end{equation}
for some $f\in\Gamma_0(\R^n).$ Specifically,
$$f(x)=\sup\limits_{v\in\R^n}\left\{\left[e_{r_1}g(x+v)-r_1q(v)\right]+\left[e_{r_2}h(x+v)-r_2q(v)\right]\right\}.$$
\end{cor}
\begin{proof}
Denote $e_{r_1}g,e_{r_2}h$ by $\bar{g},\bar{h},$ respectively. Then by Proposition \ref{1lip}, $\nabla\bar{g}$ is $r_1$-Lipschitz and $\nabla\bar{h}$ is $r_2$-Lipschitz. Hence, $\nabla\bar{f}$ is ($r_1+r_2$)-Lipschitz, where $\bar{f}=\bar{g}+\bar{h}=e_{r_1}g+e_{r_2}h.$ Applying Proposition \ref{1lip} again, we have that $\bar{f}=e_{r_1+r_2}f$ for some $f\in\Gamma_0(\R^n).$ Now to find $f,$ we apply the Fenchel conjugate to \eqref{r1plusr2}:
\begin{align*}
f^*+\frac{q}{r_1+r_2}&=(e_{r_1}g+e_{r_2}h)^*\\
f&=\left[(e_{r_1}g+e_{r_2}h)^*-\frac{q}{r_1+r_2}\right]^*.
\end{align*}
By Toland-Singer duality for the conjugate of a difference \cite[Corollary 14.19]{convmono}, we obtain that for every $x\in\R^n,$
\begin{align*}
f(x)&=\sup\limits_{v\in\R^n}\left[(e_{r_1}g+e_{r_2}h)^{**}(x+v)-\left(\frac{q}{r_1+r_2}\right)^*(v)\right]\\
&=\sup\limits_{v\in\R^n}\left[(e_{r_1}g+e_{r_2}h)(x+v)-(r_1+r_2)q(v)\right]\\
&=\sup\limits_{v\in\R^n}\left\{\left[e_{r_1}g(x+v)-r_1q(v)\right]+\left[e_{r_2}h(x+v)-r_2q(v)\right]\right\}.
\end{align*}
\end{proof}
\begin{rem}
Corollary \ref{r1r2} gives us that for $r>0$ and $f\in\Gamma_0(\R^n),$
$$f=\left[(e_rf)^*-\frac{q}{r}\right]^*.$$ Therefore, by Toland-Singer duality, for every $x\in\R^n$ we have
$$f(x)=\sup\limits_{v\in\R^n}[e_rf(x+v)-rq(v)].$$This is the Hiriart-Urruty deconvolution \cite{hiriart1994deconvolution}.
\end{rem}
\begin{prop}\label{equivalent}
Let $A\in S_+^n.$ Then the following are equivalent:
\begin{itemize}
\item[(i)] $A$ is nonexpansive, i.e. $\|Ax-Ay\|\leq\|x-y\|$ for all $x,y\in\R^n;$
\item[(ii)] $A$ is firmly nonexpansive, i.e. $\|Ax-Ay\|^2\leq\langle x-y,Ax-Ay\rangle$ for all $x,y\in\R^n;$
\item[(iii)] $A=(P+\Id)^{-1}$ for some maximally monotone linear relation $P.$
\end{itemize}
\end{prop}
\begin{proof} Denote the eigenvalues of $A$ as $\lambda_1,\lambda_2,\ldots,\lambda_n.$ Since $A\in S_+^n,$ all eigenvalues are real and nonnegative (see \cite{matanalapp} Section 7.6). \\

\noindent(i)$\Leftrightarrow$(ii) Suppose that statement (i) is true. Then, letting $z=x-y$ and squaring both sides, we have
\begin{align*}
&&\|Az\|^2&\leq\|z\|^2\\
&\Leftrightarrow&\langle z,A^\top Az\rangle&\leq\langle z,z\rangle\\
&\Leftrightarrow&\langle z,A^2z\rangle&\leq\langle z,z\rangle\\
&\Leftrightarrow&\langle z,(\Id-A^2)z\rangle&\geq0\mbox{ for all }z\in\R^n.
\end{align*}
 The inequality above is equivalent to the statement $\Id-A^2\in S_+^n,$ so $1-\lambda_i^2\geq0$ for all $i\in\{1,2,\ldots,n\}.$ Since $A\in S_+^n,$ we have $\lambda_i\geq0$ for all $i.$ Hence, statement (i) is equivalent to the following:
\begin{equation}\label{eq:lambdas}
0\leq\lambda_i\leq1\mbox{ for all }i\in\{1,2,\ldots,n\}.
\end{equation}
Now suppose that statement (ii) is true. This gives
\begin{align*}
&&\langle z,A^\top Az\rangle&\leq\langle z,Az\rangle\\
&\Leftrightarrow&\langle z,A^2z\rangle&\leq\langle z,Az\rangle\\
&\Leftrightarrow&\langle z,(A-A^2)z\rangle&\geq0\mbox{ for all }z\in\R^n.
\end{align*}
Then $(\lambda_i-\lambda_i^2)\geq0\Rightarrow\lambda_i(1-\lambda_i)\geq0$ for all $i\in\{1,2,\ldots,n\},$ so that $0\leq\lambda_i\leq1.$ Hence, statement (ii) is equivalent to \eqref{eq:lambdas}.\\

\noindent(ii)$\Leftrightarrow$(iii) Suppose that statement (ii) is true. Then Fact \ref{nonexpequivmono} gives us that $A=(\Id+P)^{-1}$ for some maximally monotone relation $P.$ Since $A$ is a matrix, we have that $A$ is linear, so that $A^{-1}$ is a linear relation. Note that the matrix inverse of $A$ may not exist; here we are referring to the general set-valued inverse of $A.$ Then we have $\Id+P=A^{-1}$ $\Rightarrow P=A^{-1}-\Id,$ so that $P$ is a linear relation. Thus statement (ii) implies statement (iii). Conversely, supposing that statement (iii) is true and applying Fact \ref{nonexpequivmono}, statement (ii) is immediately implied.
\end{proof}
Proposition \ref{equivalent} will allow us to prove Theorem \ref{nonthm}, one of the main results of this paper. Existence of a Moreau envelope is closely tied to nonexpansiveness, as the following proposition and example demonstrate, and Theorem \ref{nonthm} ultimately concludes.
\begin{prop}\label{notnon}
If $A\in S_+^n$ is not nonexpansive, then $f(x)=\frac{r}{2}\langle x,Ax\rangle+\langle b,x\rangle+c$ is not the Moreau envelope with prox-parameter $r$ of a proper, lsc, convex function.
\end{prop}
\begin{proof}
Suppose that $A$ is not nonexpansive. Then
\begin{equation}\label{contr}
\exists~x,y\in\R^n\mbox{ such that }\|Ax-Ay\|>\|x-y\|.
\end{equation}
Suppose that $f$ is the Moreau envelope with prox-parameter $r$ of some $g\in\Gamma_0(\R^n).$ Then by Theorem \ref{1lip}, $\nabla f$ is $r$-Lipschitz. That is, for all $x,y\in\R^n,$ we have
\begin{align*}
\|\nabla f(x)-\nabla f(y)\|&\leq r\|x-y\|\\
\|(rAx+b)-(rAy+b)\|&\leq r\|x-y\|\\
\|Ax-Ay\|&\leq\|x-y\|.
\end{align*}
This is a contradiction to \eqref{contr}. Therefore, $f$ is not the Moreau envelope with prox-parameter $r$ of any $g\in\Gamma_0(\R^n).$
\end{proof}
\begin{ex}\label{threes}
Let $A=\left[\begin{array}{c c}3&0\\0&3\end{array}\right],$ $f(x)=\frac{1}{2}\langle x,Ax\rangle.$ Then there does not exist $g\in\Gamma_0(\R^2)$ such that $f(x)=e_1g(x).$ However, $f(x)=e_3g(x),$ where $g(x)=\frac{1}{2}\langle x,x\rangle.$
\end{ex}
\begin{proof}
Using prox-parameter $r=1,$ we know that there cannot exist $g$ with $f=e_1g$ as a direct consequence of Proposition \ref{notnon}, since $A$ is not nonexpansive. However, rearranging the expression as $f(x)=\frac{3}{2}\langle x,\Id x\rangle$ gives a larger prox-parameter $\tilde{r}=3$ and a nonexpansive matrix $\Id,$ so there does exist $g\in\Gamma_0(\R^2)$ such that $f(x)=e_3g(x).$
\end{proof}

\subsubsection{The Moreau envelope $e_{r}f$ is linear-quadratic $\Leftrightarrow$ $f$ is generalized linear-quadratic}

In this section, we present the remaining main result of the paper, a characterization of when a convex function is a generalized linear-quadratic. It has to do with convex Moreau envelopes, and we begin with a theorem that explicitly determines the Moreau envelope of a generalized linear-quadratic function.
\begin{thm}\label{generf}
Let $A:\R^n\rightrightarrows\R^n$ be a maximally monotone symmetric linear relation. Let $a,b\in\R^n,$ $c\in\R,$ $r>0.$ Define the generalized linear-quadratic function
$$f(x)=\frac{r}{2}\langle x-a,A(x-a)\rangle+\langle b,x\rangle+c.$$ Then for every $x\in\R^n,$
$$e_rf(x)=rq_{(\Id+A^{-1})^{-1}}\left(x-a-\frac{b}{r}\right)+\langle b,x\rangle-\frac{1}{r}q(b)+c.$$
\end{thm}
\begin{proof}
By Proposition \ref{calcenv}, we have
\begin{align*}
e_rf&=re_1(f/r)=re_1(q_A(\cdot-a)+\langle\cdot,b/r\rangle+c/r)\\
&=r[e_1(q_A(\cdot-a))(\cdot-b/r)+\langle\cdot,b/r\rangle-q(b/r)+c/r]\\
&=r[q_{(\Id+A^{-1})^{-1}}(\cdot-b/r-a)+\langle\cdot,b/r\rangle-q(b/r)+c/r]\\
&=rq_{(\Id+A^{-1})^{-1}}(\cdot-a-b/r)+\langle\cdot,b\rangle-\frac{1}{r}q(b)+c.
\end{align*}
\end{proof}
\begin{cor}
Let $A:\R^n\rightrightarrows\R^n$ be a maximally monotone symmetric linear relation. Then
\begin{itemize}
\item[(i)]$e_1(q_A)=q_{(\Id+A^{-1})^{-1}},$
\item[(ii)]$q_A=e_1g$ for some $g\in\Gamma_0(\R^n)$ if and only if $A$ is nonexpansive.
\end{itemize}
\end{cor}
\begin{proof}
(i) This follows from Theorem \ref{generf} with $a=b=c=0$ and $r=1.$\\

\noindent(ii) This follows from part (i) above and Proposition \ref{1lip} with $r=1.$
\end{proof}
\begin{prop}\label{convquad}
Let $f(x)\in\Gamma_0(\R^n)$ be quadratic, i.e. $f(x)=\frac{r}{2}\langle x,Ax\rangle+\langle b,x\rangle+c,$ $A\in S_+^n,$ $b\in\R^n,$ $c\in\R,$ $r>0.$ Then $e_rf\in\Gamma_0(\R^n)$ and $e_rf$ is quadratic. Specifically,
$$e_rf(x)=\frac{r}{2}\langle x,[\Id-(\Id+A)^{-1}]x\rangle+\langle b,(\Id+A)^{-1}x\rangle-\frac{1}{2r}\langle b,(\Id+A)^{-1}b\rangle+c,$$
where $\Id-(\Id+A)^{-1}\in S_+^n.$
\end{prop}
\begin{proof}Applying Theorem \ref{generf} with $a=0$ and denoting $(\Id+A)^{-1}$ as $\C,$ we have
\begin{align*}
e_rf(x)&=rq_{(\Id+A^{-1})^{-1}}\left(x-\frac{b}{r}\right)+\langle b,x\rangle-\frac{1}{r}q(b)+c\\
&=\frac{r}{2}\left\langle x-\frac{b}{r},(\Id+A^{-1})^{-1}\left(x-\frac{b}{r}\right)\right\rangle+\langle b,x\rangle-\frac{1}{2r}\langle b,b\rangle+c\\
&=\frac{r}{2}\left\langle x-\frac{b}{r},[\Id-(\Id+A)^{-1}]\left(x-\frac{b}{r}\right)\right\rangle+\langle b,x\rangle-\frac{1}{2r}\langle b,b\rangle+c\\
&=\frac{r}{2}\langle x,x\rangle-\langle x,b\rangle+\frac{1}{2r}\langle b,b\rangle-\frac{r}{2}\langle x,\C x\rangle+\langle x,\C b\rangle-\frac{1}{2r}\langle b,\C b\rangle+\langle b,x\rangle-\frac{1}{2r}\langle b,b\rangle+c\\
&=\frac{r}{2}\langle x,(\Id-\C)x\rangle+\langle x,\C b\rangle-\frac{1}{2r}\langle b,\C b\rangle+c\\
&=\frac{r}{2}\langle x,[\Id-(\Id+A)^{-1}]x\rangle+\langle b,(\Id+A)^{-1}x\rangle-\frac{1}{2r}\langle b,(\Id+A)^{-1}b\rangle+c
\end{align*}
Since $\Id-(\Id+A)^{-1}=(\Id+A^{-1})^{-1}$ is monotone symmetric, we have that $\Id-(\Id+A)^{-1}\in S^n_+$ and the proof is complete.
\end{proof}
\begin{thm}\label{nonthm}
Let $f$ be a convex quadratic function: $f(x)=\frac{r}{2}\langle x,Ax\rangle+\langle b,x\rangle+c,$ $A\in S^n_+,$ $b\in\R^n,$ $c\in\R,$ $r>0.$ Then $A$ is nonexpansive if and only if $f=e_rg$ where $g$ is a generalized linear-quadratic function: $$g(x)=\begin{cases}\frac{r}{2}\langle x,P^{-1}x\rangle+\langle t,x\rangle+s,&\mbox{if }x\in\dom P^{-1},\\\infty,&\mbox{if }x\not\in\dom P^{-1},\end{cases}$$ with $P^{-1}$ a monotone linear relation. This includes the case $g(x)=\iota_{\{t\}}(x)+s.$ Specifically, $g$ is as follows.
\begin{itemize}
\item[(i)] If $A=\Id,$ then $g(x)=\iota_{\{t\}}(x)+s=q_{N_{\{0\}}}(x-t)+s$ (see Remark \ref{normalrem}), where
\begin{equation}\label{quad1}
t=-\frac{b}{r}\mbox{ and }s=c-\frac{r}{2}\langle b,b\rangle.
\end{equation}
\item[(ii)] The matrix $A\in S^n_+\setminus\Id$ is nonexpansive if and only if $g(x)=\frac{r}{2}\langle x,P^{-1}x\rangle+\langle t,x\rangle+s,$ where
\begin{equation}\label{quad2}
P^{-1}=(\Id-A)^{-1}-\Id,~t\in (\Id-A)^{-1}b,\mbox{ and }s=c+\frac{1}{2r}\langle b,(\Id-A)^{-1}b\rangle.
\end{equation}
\end{itemize}
\end{thm}
\begin{proof}
(i) Let $g(x)=\iota_{\{t\}}(x)+s=\begin{cases}s,&x=t,\\\infty,&x\neq t.\end{cases}$ Then
\begin{align*}
e_rg(x)&=\inf\limits_y\left\{g(y)+\frac{r}{2}\|y-x\|^2\right\}\\
&=g(t)+\frac{r}{2}\|t-x\|^2\\
&=s+\frac{r}{2}\langle t-x,t-x\rangle\\
&=s+\frac{r}{2}(\langle t,t\rangle-2\langle t,x\rangle+\langle x,x\rangle)\\
&=\frac{r}{2}\langle x,\Id x\rangle-r\langle t,x\rangle+\frac{r}{2}\langle t,t\rangle+s.
\end{align*}
Equating
\begin{equation}\label{quad3}
A=\Id,~~b=-rt,~~\mbox{and }c=\frac{r}{2}\langle t,q\rangle+s,
\end{equation}
we have that for any choice of $b\in\R^n$ and $c\in\R,$ there exists $g(x)=\iota_{\{t\}}(x)+s$ such that
$$f(x)=\frac{r}{2}\langle x,x\rangle+\langle b,x\rangle+c=e_rg(x).$$
The equations in \eqref{quad1} are obtained by solving the equations in \eqref{quad3} for $t$ and $s.$\\

\noindent(ii) By Proposition \ref{notnon}, if $A$ is not nonexpansive, then there does not exist $g\in\Gamma_0(\R^n)$ such that $f(x)=e_rg(x).$ Thus, supposing that there does exist such a $g,$ we have that $A$ is nonexpansive. Then by Fact \ref{baillon}, $A=(\Id+P)^{-1}$ for some maximally monotone operator $P.$ Since $A\in S_+^n,$ $P$ is a symmetric linear relation by Proposition \ref{equivalent}. Now using the general set-valued inverse $P^{-1},$ we set
$$g(x)=\begin{cases}
\frac{r}{2}\langle x,P^{-1}x\rangle+\langle t,x\rangle+s,&\mbox{if }x\in\dom P^{-1},\\\infty,&\mbox{if }x\not\in\dom P^{-1}.
\end{cases}$$
This function $g$ is well-defined due to Fact \ref{liangprop}. Since $P$ is a monotone linear relation, the function
$$h(x)=\begin{cases}
\frac{1}{2}\langle x,P^{-1}x\rangle,&\mbox{if }x\in\dom P,\\\infty,&\mbox{if }x\not\in\dom P
\end{cases}$$ is single-valued. Then by Proposition \ref{convquad}, we have
\begin{align*}
e_rg(x)&=\frac{r}{2}\langle x,[\Id-(\Id+P^{-1})^{-1}]x\rangle+\langle t,(\Id+P^{-1})^{-1}x\rangle-\frac{1}{2r}\langle t,(\Id+P^{-1})^{-1}t\rangle+s\\
&=\frac{r}{2}\langle x,(\Id+P)^{-1}x\rangle+\langle t,(\Id+P^{-1})^{-1}x\rangle-\frac{1}{2r}\langle t,(\Id+P^{-1})^{-1}t\rangle+s.\mbox{ (Fact \ref{rockwets12.14})}
\end{align*}
Equating
\begin{equation}\label{pqr}
A=(\Id+P)^{-1},~~b=(\Id+P^{-1})^{-1}t,\mbox{ and }c=s-\frac{1}{2r}\langle t,(\Id+P^{-1})^{-1}t\rangle,
\end{equation}
we have that $f(x)=\frac{r}{2}\langle x,Ax\rangle+\langle b,x\rangle+c$ is the Moreau envelope of $g(x)=\frac{r}{2}\langle x,P^{-1}x\rangle+\langle t,x\rangle+s.$ The equations in \eqref{quad2} are obtained by solving the equations in \eqref{pqr} for $P^{-1},$ $t,$ and $s.$
\end{proof}
\begin{thm}\label{genquadthm}
The function $f\in\Gamma_0(\R^n)$ is a generalized linear-quadratic function if and only if $e_rf\in\Gamma_0(\R^n)$ is a quadratic function. Specifically, $$e_rf(x)=\frac{1}{2}\langle x,Ax\rangle+\langle b,x\rangle+c\qquad\forall x\in\R^n$$with $A\in S^n_+,$ if and only if
$$f(x)=q_{(A^{-1}-\Id/r)^{-1}}\left(x+\frac{b}{r}\right)+\langle b,x\rangle+c+\frac{1}{2r}\|b\|^2\qquad\forall x\in\R^n.$$
\end{thm}
\begin{proof}
\begin{itemize}
\item[$(\Rightarrow)$] This is the statement of Theorem \ref{generf}.
\item[$(\Leftarrow)$] Let $e_rf(x)=\frac{1}{2}\langle x,Ax\rangle+\langle b,x\rangle+c,$ with $A$ symmetric, linear and monotone, $b\in\R^n,$ $c\in\R.$ Then$$(e_rf)^*=f^*+\frac{1}{r}q,$$and$$(q_A+\langle\cdot,b\rangle+c)^*=q_{A^{-1}}(\cdot-b)-c.$$ This gives us that
$$f^*=q_{A^{-1}-\Id/r}(\cdot-b)-\langle\cdot,b/r\rangle-c+q(b)/r.$$It follows that
\begin{align*}
f&=(q_{A^{-1}-\Id/r}(\cdot-b))^*(\cdot+b/r)+c-q(b)/r\\
&=q_{(A^{-1}-\Id/r)^{-1}}(\cdot+b/r)+\langle\cdot+b/r,b\rangle+c-q(b)/r\\
&=q_{(A^{-1}-\Id/r)^{-1}}(\cdot+b/r)+\langle\cdot+b/r,b\rangle+c-q(b)/r\\
&=q_{(A^{-1}-\Id/r)^{-1}}(\cdot+b/r)+\langle\cdot,b\rangle+c+q(b)/r.
\end{align*}
Thus, $f\in\Gamma_0(\R^n)$ is a generalized linear-quadratic function.
\end{itemize}
\end{proof}

\section{Applications}\label{sec:semi}

This final section presents a few applications of the theory seen thus far. We build on the idea of extended norms, give an application to the least squares problem, and explore the limit of a sequence of generalized linear-quadratic functions.

\subsection{A seminorm with infinite values}

In \cite{beer2015norms,beer2015structural}, Beer and Vanderwerff present the idea of norms that are allowed to take on infinite values. These so-called extended norms are functions on linear spaces that satisfy the properties of a norm when they are finite-valued, but can be infinite-valued as well. The authors extend many properties of norms to the setting of an extended norm space $(X,\|\cdot\|),$ where $X$ is a vector space and $\|\cdot\|$ is an extended norm. In that spirit, we present here an extended seminorm.
\begin{df}
A function $k:\R^n\rightarrow\OR$ is a \emph{gauge} if $k$ is a nonnegative, positively homogeneous, convex function such that $k(0)=0.$ Thus, a gauge is a function $k$ such that
$$k(x)=\inf\{\mu\geq0: \ x\in\mu C\}$$for some nonempty convex set $C.$
\end{df}
\begin{df}
The \emph{polar} of a gauge $k$ is the function $k^o$ defined by
$$k^o(x^*)=\inf\{\mu^*\geq0:\ \langle x,x^*\rangle\leq\mu^*k(x)~\forall x\in\R^n\}.$$If $k$ is finite everywhere and positive except at the origin, the polar of $k$ can be written as
$$k^o(x^*)=\sup\left\{\frac{\langle x,x^*\rangle}{k(x)}: x\neq 0\right\}.$$
\end{df}
\begin{df}
A function $k:\R^n\rightarrow\OR$ is an \emph{extended seminorm} if
\begin{itemize}
\item[(i)]$k(x)\geq0~\forall x\in\R^n,$
\item[(ii)]$k(\alpha x)=|\alpha|k(x)~\forall x\in\R^n,\forall\alpha\in\R,$
\item[(iii)]$k(x+y)\leq k(x)+k(y)~\forall x,y\in\R^n,$
\item[(iv)]$k(x)=\infty$ if $x\not\in\dom k.$
\end{itemize}
\end{df}
\begin{thm}
Let $A$ be a maximally monotone symmetric linear relation. Then the following hold.
\begin{itemize}
\item[(i)] The function$$k=\left(2q_A\right)^{1/2}$$is an extended seminorm. Moreover, \begin{equation}\label{k0A0}k^{-1}(0)=A^{-1}0.\end{equation}
\item[(ii)]For all $x\in\dom A$ and for all $x^*\in\ran A,$ we have $$\langle x,x^*\rangle\leq\sqrt{\langle x,Ax\rangle}\sqrt{\langle x^*,A^{-1}x^*\rangle}.$$
\item[(iii)] The closed convex sets$$C=\{x:q_A(x)\leq1\},\qquad C^*=\{x^*:q_{A^{-1}}(x^*)\leq1\}$$are polar to each other.
\end{itemize}
\end{thm}
\begin{proof}
(i) Applying \cite[Corollary 15.3.1]{convanalrock} with $f=q_A$ and $p=2,$ We have that $k$ is a gauge function. Thus, $k$ is an extended seminorm. To see \eqref{k0A0}, we have that $k(x)=0\Leftrightarrow q_A(x)=0,$ so it suffices to apply Proposition \ref{tfae5}.\\

\noindent(ii) By Proposition \ref{qastartheorem}, $q_A^*=q_{A^{-1}}.$ By \cite[Corollary 15.3.1]{convanalrock}, we have that $k^o(x^*)=(2q^*_A(x^*))^{1/2},$ and that $\forall x\in\dom A,\forall x^*\in\ran A,$
\begin{align*}
\langle x,x\rangle^*&\leq k(x)k^o(x^*)\\
&=2(q_A(x))^{1/2}(q_{A^{-1}}(x^*))^{1/2}\\
&=\sqrt{\langle x,Ax\rangle}\sqrt{\langle x^*,A^{-1}x^*\rangle}.
\end{align*}
(iii) By \cite[Corollary 15.3.2]{convanalrock}, we have that the closed, convex sets
$$C=\{x:\langle x,Ax\rangle\leq1\},\qquad C^*=\{x^*:\langle x^*,A^{-1}x^*\rangle\leq1\}$$are polar to each other.
\end{proof}
\begin{rem}
The above result generalizes Rockafellar's result on \cite[p. 136]{convanalrock} with $Q=A$ from a positive definite matrix to a maximally monotone symmetric linear relation.
\end{rem}

\subsection{The least squares problem}

In this section, we show that generalized quadratic functions can be used to study the least squares problem.
Let $A\in \R^{m\times n}$ and $b\in\R^{m}$. The general least squares problem is to find a
vector $x\in\R^n$ that
minimizes
\begin{equation}\label{e:least}
\ell:\R^n\rightarrow\R: x\mapsto \frac{1}{2}\|Ax-b\|^2=q_{A^\top A}(x)-\langle x,A^\top b\rangle+q_{\Id}(b).
\end{equation}
\begin{thm} For the function $\ell$ given by \eqref{e:least}, we have
\begin{enumerate}
\item[(i)]\label{i:conj} $\ell^*(y)=q_{(A^\top A)^{-1}}(y+A^\top b)-q_{\Id}(b)\quad\forall y\in\R^n;$
\item[(ii)]\begin{equation}\label{e:subdiff}
\partial \ell^*(y)=(A^\top A)^{-1}(y+A^\top b)\quad\forall y\in\R^n,
\end{equation}
 and
\begin{equation}\label{e:domain}
\dom \ell^*=\ran A^\top.
\end{equation}
\end{enumerate}
\end{thm}
\begin{proof}
(i) Apply Example~\ref{rockext}.\\

\noindent(ii) Apply Proposition~\ref{fourprops} to obtain \eqref{e:subdiff}.
To see \eqref{e:domain}, using the facts that $\ran A^\top A=\ran A^\top$ (c.f. \cite[page 212]{matanalapp}) and that
$\ran A^\top$ is a subspace, we have
$$\dom \ell^* =\dom[(A^\top A)^{-1}-A^\top b]=\ran(A^\top A-A^\top b)=\ran(A^\top-A^\top b)=\ran A^\top.$$
\end{proof}

\subsection{Permanently staying in the generalized linear-quadratic world}

We end this work with an application for sequences of $q_{A_k}$ functions with $A_k$ linear relations, and the development of a calculus for the generalized linear-quadratic functions.
\begin{prop}[epiconvergence]\label{thm:sequence}
\begin{itemize}
\item[(i)]For all $k\in\N,$ let\begin{equation}\label{eq:qak}f_k=q_{A_k}(\cdot-a_k)+\langle b_k,\cdot\rangle+c_k,\end{equation}where $A_k$ is a maximally monotone symmetric linear relation, $a_k,b_k\in\R^n,$ $c_k\in\R.$ Suppose that $f_k\overset{e}\rightarrow f$ and that $f$ is proper. Then $f$ is a generalized linear-quadratic function:\begin{equation}\label{eq:sequence}f=q_{A}(\cdot-a)+\langle b,\cdot\rangle+c,\end{equation}where $A$ is a maximally monotone symmetric linear relation, $a,b\in\R^n,$ $c\in\R.$
\item[(ii)]For all $k\in\N,$ let\begin{equation}\label{eq:qak2}f_k=q_{A_k}+c_k,\end{equation}where $A_k$ is a maximally monotone symmetric linear relation, $c_k\in\R.$ Suppose that $f_k\overset{e}\rightarrow f$ and that $f$ is proper. Then $f$ is a generalized linear-quadratic function:\begin{equation}\label{eq:sequence2}f=q_{A}+c,\end{equation}where $A$ is a maximally monotone symmetric linear relation, $c\in\R.$
\end{itemize}
\end{prop}
\begin{proof}(i) As $f_k\overset{e}\rightarrow f,$ we have $\partial f_k\overset{g}\rightarrow\partial f.$ Differentiating \eqref{eq:qak}, we find that $\partial f_k=A_k(\cdot-a_k)+b_k,$ so that $\gra\partial f_k=\gra A_k+(a_k,b_k)$ is maximally monotone and affine. Thus, $\gra\partial f$ is maximally monotone and affine. By \cite[Theorem 4.3]{bauschke2010borwein},  $\gra\partial f=\gra A+(a,b)$ for some maximally monotone symmetric linear relation $A.$ Then by Proposition \ref{fourprops}, we have $A=\partial q_A,$ so that \eqref{eq:sequence} is true.\\

\noindent(ii) The proof is similar to that of part (i), except that differentiating \eqref{eq:qak2} we find that $\partial f_k=A_k~\forall x\in\R^n,$ so that $\gra\partial f_k=\gra A_k$ is a linear subspace. Thus, $\gra\partial f$ is a linear subspace, $\gra\partial f=\gra A$ for some maximally monotone symmetric linear relation $A,$ and Proposition \ref{fourprops} gives $A=\partial q_A$ so that \eqref{eq:sequence2} is true.
\end{proof}
\noindent As a result of Proposition \ref{thm:sequence}, we are now able to define calculus rules for generalized linear-quadratic functions. We do so in the form of Theorems \ref{thm:qacalculus} and \ref{cor:calculusgen}, for which we define the following sets.
\begin{df}
Denote by $\mathcal{A}$ the set of maximally monotone symmetric linear relations on $\R^n.$ We define $S$ as the set of convex generalized linear-quadratic functions:
$$S=\left\{f=q_A(\cdot-a)+\langle b,\cdot\rangle+c:A\in\mathcal{A},a,b\in\R^n,c\in\R,f\in\Gamma_0(\R^n)\right\}.$$We define $T$ as the subset of $S$ obtained by setting $a=0:$
$$T=\left\{f=q_A+\langle b,\cdot\rangle+c:A\in\mathcal{A},b\in\R^n,c\in\R,f\in\Gamma_0(\R^n)\right\}.$$We define $U$ as the subset of $S$ obtained by setting $a=b=0:$
$$U=\left\{f=q_A+c:A\in\mathcal{A},c\in\R,f\in\Gamma_0(\R^n)\right\}.$$
\end{df}
\noindent We begin with calculus rules for the simpler case, the set $U.$
\begin{thm}\label{thm:qacalculus}
Let $d$ be the Attouch-Wets metric (see Definition \ref{attouchwetsmetric}.) The following hold.
\begin{itemize}
\item[(i)]The metric space $(U,d)$ is complete.
\item[(ii)]If $f\in U,$ then $f^*\in U.$
\item[(iii)]If $f\in U$ and $\lambda>0,$ then $\lambda f\in U.$
\item[(iv)]If $f_1,f_2\in U,$ then $f_1+f_2\in U.$
\item[(v)]If $f_1,f_2\in U,$ then $f_1\oblong f_2\in U.$
\end{itemize}
\end{thm}
\begin{proof}
(i) This follows from Proposition \ref{thm:sequence}.\\

\noindent (ii) Let $f\in U,$ $f=q_A+c.$ By Proposition \ref{qastartheorem}, we have
$$f^*=(q_A+c)^*=q_A^*-c=q_{A^{-1}}-c,$$
which is a convex generalized linear-quadratic function of the required form. Therefore, $f^*\in U.$\\

\noindent(iii) It is clear that $f=q_A+c\in U$ and $\lambda>0$ yields $\lambda f=q_{\lambda A}+\lambda c\in U.$\\

\noindent(iv) Let $f_1,f_2\in U,$ $f_1=q_{A_1}+c_1,$ $f_2=q_{A_2}+c_2.$ By Proposition \ref{p:difference}, we have that $$(f_1+f_2)=q_{A_1+A_2}+c_1+c_2$$is a convex generalized linear-quadratic function of the form found in $U.$ Therefore, $f_1+f_2\in U.$\\

\noindent(v) Let $f_1,f_2\in U,$ $f_1=q_{A_1}+c_1,$ $f_2=q_{A_2}+c_2.$ By Propositions \ref{parallelsum} and \ref{qastartheorem}, we have $$f_1\oblong f_2=q^*_{A_1^{-1}+A_2^{-1}}+c_1+c_2=q_{(A_1^{-1}+A_2^{-1})^{-1}}+c_1+c_2\in U.$$
\end{proof}
\noindent For the more general setting of the sets $S$ and $T,$ the calculus rules are not so straightforward. More stringent conditions are necessary; the following theorem provides the obtainable results.
\begin{thm}\label{cor:calculusgen}
Let $d$ be the Attouch-Wets metric (see Definition \ref{attouchwetsmetric}). The following hold.
\begin{itemize}
\item[(i)]The metric space $(S,d)$ is complete.
\item[(ii)]If $f\in S,$ then $f^*\in S.$
\item[(iii)]If $f\in S$ with $b=0,$ then $f^*\in T.$
\item[(iv)]If $f\in S$ ($f\in T$) and $\lambda>0,$ then $\lambda f\in S$ ($\lambda f\in T$).
\item[(v)]If $f_1,f_2\in T,$ then $f_1+f_2\in T.$
\end{itemize}
\end{thm}
\begin{proof}
(i) This follows from Proposition \ref{thm:sequence}.\\

\noindent(ii) Let $f\in S,$ $f=q_A(\cdot-a)+\langle b,\cdot\rangle+c.$ Combining Proposition \ref{qastartheorem} and \cite[Proposition 13.20]{convmono}, we have
\begin{align*}
f^*&=(q_A(\cdot-a)+\langle b,\cdot\rangle+c)^*\\
&=(q_A(\cdot-a))^*(\cdot-b)-c\\
&=(q_{A^{-1}}+\langle a,\cdot\rangle)(\cdot-b)-c\\
&=q_{A^{-1}}(\cdot-b)+\langle a,\cdot-b\rangle-c,
\end{align*}
which is a convex generalized linear-quadratic function. Therefore, $f^*\in S.$\\

\noindent(iii) Let $f\in S,$ $f=q_A(\cdot-a)+c.$ By the same procedure as in the proof of (ii), we have
$$f^*=q_{A^{-1}}+\langle a,x\rangle-c\in T.$$

\noindent(iv) It is clear that $f=q_A(\cdot-a)+\langle b,\cdot\rangle+c\in S$ and $\lambda>0$ yields $\lambda f=q_{\lambda A}(\cdot-a)+\lambda\langle b,\cdot\rangle+\lambda c\in S,$ and that $a=0$ yields $\lambda f\in T.$\\

\noindent(v) Let $f_1,f_2\in T,$ $f_1=q_{A_1}+\langle b_1,\cdot\rangle+c_1,$ $f_2=q_{A_2}+\langle b_2,\cdot\rangle+c_2.$ By Proposition \ref{p:difference}, we have that $$(f_1+f_2)=q_{A_1+A_2}+\langle b_1+b_2,\cdot\rangle+c_1+c_2$$is a convex generalized linear-quadratic function. Therefore, setting $f=f_1+f_2,$ $A=A_1+A_2,$ $b=b_1+b_2$ and $c=c_1+c_2,$ we have that $f=q_A+\langle b,\cdot\rangle+c\in T.$
\end{proof}

\section{Conclusion}\label{sec:conc}

On $\R^n,$ the Moreau envelope of a generalized linear-quadratic function was explicitly identified. Conversely, it was determined under what conditions a quadratic function is a Moreau envelope of a generalized linear-quadratic. Characterizations of the existence of the Moreau envelope for convex functions involving Lipschitz continuity and nonexpansiveness were established, and we showed that a convex function is generalized linear-quadratic if and only if its Moreau envelope is convex linear-quadratic. The topic of generalized linear-quadratic functions was discussed at length; several useful characterizations and properties were established. We gave applications to generalized seminorms, the least squares problem, the epilimit of a sequence and the calculus of generalized linear-quadratic functions.

\section*{Acknowledgements}

Chayne Planiden was supported by UBC University Graduate Fellowship and by Natural Sciences and Engineering Research Council of Canada. Xianfu Wang was partially supported by a Natural Sciences and Engineering Research Council of Canada Discovery Grant.

\bibliographystyle{plain}
\bibliography{Bibliography}{}
\end{document}